\documentclass[11pt, oneside, reqno]{amsart}
\usepackage[a4paper,margin=1in]{geometry}
\usepackage[T1]{fontenc}
\usepackage{amsfonts}
\usepackage{amssymb}
\usepackage{amsmath}
\usepackage{amsthm}
\usepackage{mathtools}
\usepackage{mathrsfs}
\usepackage{zref-clever}
\usepackage{hyperref}
\usepackage{stmaryrd}
\usepackage{tikz-cd}
\usepackage{enumitem}
\usepackage{multirow}
\usepackage{bbm}
\usepackage{comment}
\usepackage{float}
\usepackage{url}
\usepackage[en-GB]{datetime2}
\DTMlangsetup[en-GB]{ord=omit}
\allowdisplaybreaks

\numberwithin{equation}{section}

\zcsetup{cap,nameinlink=false}

\AddToHook{env/lemma/begin}{%
	\zcsetup{countertype={theorem=lemma}}}
\AddToHook{env/proposition/begin}{%
	\zcsetup{countertype={theorem=proposition}}}
\AddToHook{env/corollary/begin}{%
	\zcsetup{countertype={theorem=corollary}}}
\AddToHook{env/conjecture/begin}{%
	\zcsetup{countertype={theorem=conjecture}}}
\AddToHook{env/remark/begin}{%
	\zcsetup{countertype={theorem=remark}}}
\AddToHook{env/example/begin}{%
	\zcsetup{countertype={theorem=example}}}

\zcRefTypeSetup{question}{
	Name-sg = Question ,
	name-sg = question ,
	Name-pl = Questions ,
	name-pl = questions ,
}

\zcRefTypeSetup{conjecture}{
	Name-sg = Conjecture ,
	name-sg = conjecture ,
	Name-pl = Conjectures ,
	name-pl = conjectures ,
}

\zcRefTypeSetup{observation}{
	Name-sg = Observation ,
	name-sg = observation ,
	Name-pl = Observations ,
	name-pl = observations ,
}

\zcRefTypeSetup{speculation}{
	Name-sg = Speculation ,
	name-sg = speculation ,
	Name-pl = Speculations ,
	name-pl = speculations ,
}

\zcRefTypeSetup{remark}{
	Name-sg = Remark ,
	name-sg = remark ,
	Name-pl = Remarks ,
	name-pl = remarks ,
}

\zcRefTypeSetup{example}{
	Name-sg = Example ,
	name-sg = example ,
	Name-pl = Examples ,
	name-pl = examples ,
}

\newtheorem{theorem}{Theorem}[section]
\newtheorem{lemma}[theorem]{Lemma}
\newtheorem{proposition}[theorem]{Proposition}
\newtheorem{corollary}[theorem]{Corollary}
\newtheorem{conjecture}[theorem]{Conjecture}

\theoremstyle{definition}
\newtheorem{example}[theorem]{Example}

\theoremstyle{remark}
\newtheorem{remark}[theorem]{Remark}

\def\MR#1{}

\newcommand\cl[1]{\mathcal{#1}}
\newcommand\fr[1]{\mathfrak{#1}}
\newcommand\ov[1]{\overline{#1}}

\newcommand\wt[1]{\widetilde{#1}}
\newcommand\wh[1]{\widehat{#1}}

\newcommand\bo[1]{\boldsymbol{#1}}

\newcommand{\Q}{\mathbb{Q}}
\newcommand{\Z}{\mathbb{Z}}

\newcommand{\C}{\mathbb{C}}

\newcommand{\F}{\mathbb{F}}

\newcommand{\eps}{\varepsilon}

\newcommand{\clh}{\mathcal{H}}

\newcommand{\clo}{\mathcal{O}}

\newcommand{\eq}{\;=\;}
\newcommand{\deq}{\;:=\;}
\newcommand\req[1]{\;\stackrel{\mathclap{\text{\tiny #1}}}{=}\;}
\newcommand\requiv[1]{\;\stackrel{\mathclap{\text{\tiny #1}}}{\equiv}\;}
\def\thin{{\hspace{1pt}}}

\newcommand{\divides}{\thin|\thin}

\newcommand{\N}{\mathrm{N}}
\newcommand{\D}{\mathrm{D}}

\newcommand{\SL}{\mathrm{SL}}

\newcommand{\Tr}{\mathrm{Tr}}

\newcommand{\Res}{\mathrm{Res}}
\newcommand{\Ind}{\mathrm{Ind}}
\newcommand{\dr}{\mathrm{dR}}

\newcommand{\Sym}{\mathrm{Sym}}

\newcommand{\et}{\mathrm{\acute{e}t}}

\newcommand{\mero}{\mathrm{mero}}
\newcommand{\Span}{\mathrm{Span}}
\newcommand{\Frob}{\mathrm{Frob}}
\newcommand{\Trun}{\mathrm{Trun}}
\newcommand{\prim}{\mathrm{prim}}
\newcommand{\dif}{\mathrm{d}}
\newcommand{\aux}{\mathrm{aux}}
\newcommand{\MMF}{\mathrm{MMF}}

\renewcommand{\Im}{\operatorname{Im}}

\newmuskip\pFqskip
\mathchardef\pFcomma=\mathcode`, 

\newcommand*\pFq[5]{%
	\begingroup
	\begingroup\lccode`~=`,
	\lowercase{\endgroup\def~}{\mkern\pFqskip}%
	\mathcode`,=\string"8000
	{}_{#1}F_{#2}\biggl[\genfrac..{0pt}{}{#3}{#4};#5\biggr]%
	\endgroup
}

\newcommand*\smallpFq[5]{%
	\begingroup
	\begingroup\lccode`~=`,
	\lowercase{\endgroup\def~}{\mkern\pFqskip}%
	\mathcode`,=\string"8000
	{}_{#1}F_{#2}\bigl[\genfrac..{0pt}{}{#3}{#4};#5\bigr]%
	\endgroup
}

\title[Elliptic curves and Fourier coefficients of meromorphic modular forms]{Elliptic curves and Fourier coefficients of \\ meromorphic modular forms}
\author{Pengcheng Zhang}
\address{Max Planck Institute for Mathematics, Vivatsgasse 7, 53111 Bonn, Germany}
\email{pzhang@mpim-bonn.mpg.de}
\date{\today}

\begin{document}

\begin{abstract}
We discuss several (conjectural) congruences satisfied by the coefficients of meromorphic modular forms, or equivalently, the $p$-adic behaviors of meromorphic modular forms under the $U_p$ operator, all of which should be viewed as explicit $p$-adic manifestations of the motivic nature of meromorphic modular forms. In the generic case, we observe the connection to symmetric powers of elliptic curves via Atkin--Swinnerton-Dyer congruences, while in the CM case, we furthermore observe the connection to the $p$-adic analogue of the Chowla--Selberg periods via supercongruences. Along with the discussions, we also provide some heuristic explanations for these congruences as well as prove some of them using hypergeometric functions and the Borcherds--Shimura lift.
\end{abstract}

\maketitle

\setcounter{tocdepth}{1}
\tableofcontents

\section{Introduction}
\label{introduction}

Coefficients of modular forms have always been one of the central topics in the field of modular forms. The vast majority of work in this field is restricted to the study of holomorphic or weakly holomorphic modular forms. In this paper, we instead turn to \emph{meromorphic} modular forms. The main objects of interest are meromorphic modular forms of level $1$ and weight $k\geq 4$ with precisely one pole at a non-cuspidal point.

In the first half of this paper, we will present the conjectural\footnote{It should be noted that while this paper heavily adopts the word `conjectures', they should be more accurately named as numerical observations, as they are summarized from numerical experiments with these meromorphic modular forms.} $p$-adic properties possessed by these meromorphic modular forms as well as the heuristic reasons for those properties. In particular, their $p$-adic behaviors provide an explicit link between meromorphic modular forms of weight $k$ and symmetric $(k-2)$-nd powers of elliptic curves, which can be viewed as the $p$-adic counterpart of the same connection observed in Brown--Fonseca \cite{brown-fonseca} in the archimedean setting. We will also discuss the CM case in detail, namely when the pole is at a CM point, which may be the most interesting part of this paper and deserves thorough study in its own right. In this case, we will construct a family of meromorphic modular forms whose conjectural $p$-adic behaviors can be characterized via supercongruences and the $p$-adic analogue of the Chowla--Selberg periods.

In the second half of this paper, we will demonstrate some of the $p$-adic properties discussed in the first half. In particular, we will prove a mod $p$ congruence of meromorphic modular forms with simple poles via a hypergeometric reinterpretation, and we will show the magnetic property and the supercongruences of a particular type of meromorphic modular forms with CM poles using the Borcherds--Shimura lift, generalizing the results of Li--Neururer \cite{li-neururer} and Pa\c{s}ol--Zudilin \cite{pasol-zudilin}.

\subsection{The case of weight $4$}

We first motivate the subject through the case of weight $4$. We start with the simplest case of having a simple pole, more precisely modular forms of the form $\frac{E_4}{j-c}$ for some $c\in\Q$. We first consider the cases when $c\in\{0,1728\}$, i.e., $\frac{E_4}{j}$ and $\frac{E_4}{j-1728}$. Here we make a definition that a (general) modular form $\sum_{n=1}^\infty a_nq^n\in\Z[[q]]$ is called \emph{magnetic} if $n\divides a_n$ for all $n\in\Z^+$, and specifically \emph{$r$-magnetic} for some $r\in\Z_{\geq0}$ if $n^r\divides a_n$ for all $n\in\Z^+$. We then have the following result by Li--Neururer \cite[Theorem~1.5]{li-neururer} and Pa\c{s}ol--Zudilin \cite[Theorem~1]{pasol-zudilin}. 

\begin{theorem}
\label{1-magnetic-k=4-j=0-1728}
Both $\frac{E_4}{j}$ and $\frac{E_4}{j-1728}$ are $1$-magnetic. 
\end{theorem}

The first numerically observed example of a magnetic modular form appeared in Broadhurst--Zudilin \cite[Conjecture~4]{broadhurst-zudilin} with motivation coming from the Hall effect in electromagnetism. Li and Neururer \cite{li-neururer} proved the magnetic property of that example using the Borcherds--Shimura lift, and Pa\c{s}ol and Zudilin \cite{pasol-zudilin} generalized this technique to show the magnetic property of other meromorphic modular forms, including $\frac{E_4}{j}$ and $\frac{E_4}{j-1728}$, all of which have poles at CM points. Magnetic modular forms also occur in the context of Calabi--Yau varieties, as discussed in B\"{o}nisch--Duhr--Maggio \cite{boenisch-duhr-maggio} and B\"{o}nisch--Golyshev--Klemm \cite{boenisch-golyshev-klemm}. It is worth noting that in the case of holomorphic modular forms, it is a folklore conjecture that there does not exist any magnetic example \cite{pasol-zudilin}. This is also related to the conjecture of the infinitude of ordinary primes that is commonly attributed to Serre \cite[\#133]{serre-ordinary-conj}.

By modifying the method of using the Borcherds--Shimura lift, one can also show that both $\frac{E_4}{j}$ and $\frac{E_4}{j-1728}$ satisfy the following supercongruences\footnote{Here the word `supercongruences' refers to the type of congruences with modulus $p^{rl}$ for some $r>1$.}, which almost imply their $1$-magnetic property and in particular confirm that they are not $2$-magnetic. 

\begin{theorem}
\label{supercongruence-k=4-j=0-1728}
For all primes $p\geq 5$ and all $n,l\in\Z^+$,
\begin{alignat*}{2}
    a_{np^l}\bigg(\frac{E_4}{j}\bigg)&\;\equiv\;\big(\tfrac{-3}{p}\big)p\thin a_{np^{l-1}}\bigg(\frac{E_4}{j}\bigg)&&\pmod{p^{3l}} \\
    a_{np^l}\bigg(\frac{E_4}{j-1728}\bigg)&\;\equiv\;\big(\tfrac{-4}{p}\big)p\thin a_{np^{l-1}}\bigg(\frac{E_4}{j-1728}\bigg)&&\pmod{p^{3l}},
\end{alignat*}
where $(\frac{\cdot}{p})$ denotes the Legendre symbol.
\end{theorem}

When $c\notin\{0,1728\}$, one observes something different, with elliptic curves coming into play. Let $C/\Q$ be an elliptic curve. For a good prime $p$ of $C$, define
\begin{align*}
    a_p(C)\deq p+1-|C(\F_p)|.
\end{align*}

\begin{theorem}
\label{k=4-good-prime-theorem}
Let $C/\Q$ be an elliptic curve with $j(C)\notin\{0,1728\}$ and let $p$ be a good prime of~$C$ with $p\geq 5$ and $v_p(j(C))=0=v_p(j(C)-1728)$. Then,
\begin{align*}
    a_p\bigg(\frac{E_4}{j-j(C)}\bigg)\;\equiv\;a_p(C)^2\pmod{p}.
\end{align*}
\end{theorem}

Here the appearance of $a_p(C)^2$ already hints at the link between $\frac{E_4}{j-j(C)}$ and $\Sym^2C$, since the trace of Frobenius at $p$ on $\Sym^2C$ is $a_p(C)^2-p$, which is congruent to $a_p(C)^2$ modulo $p$. When $p$ is a supersingular prime of $C$, i.e., $p\divides a_p(C)$, \zcref{k=4-good-prime-theorem} \emph{a priori} only implies that $a_p(\frac{E_4}{j-j(C)})\equiv 0\pmod{p}$, while in fact numerical observations suggest the following stronger congruence.

\begin{conjecture}
\label{k=4-supersingular-prime-conjecture}
Let $C/\Q$ be an elliptic curve with $j(C)\notin\{0,1728\}$ and let $p$ be a supersingular prime of~$C$ with $v_p(j(C))=0=v_p(j(C)-1728)$. Then,
\begin{align*}
    a_p\bigg(\frac{E_4}{j-j(C)}\bigg)\;\equiv\;0\pmod{p^2}.
\end{align*}
\end{conjecture}

In order to draw the comparison between meromorphic modular forms of weight $4$ and symmetric squares of elliptic curves, it turns out that the `correct' space to consider on the modular form side is
\begin{align*}
	\MMF_{4,C}\deq\Span\bigg\{\frac{E_4}{j-j(C)},\;\frac{E_4}{(j-j(C))^2},\;\frac{E_4}{(j-j(C))^3}\bigg\}
\end{align*}
(for $j(C)\notin\{0,1728\}$). Now, both $\MMF_{4,C}$ and $\Sym^2C$ are $3$-dimensional, and one may expect a connection between the $U_p$-action on $\MMF_{4,C}$ and the Frobenius action at $p$ on $\Sym^2C$. Indeed, on the elliptic curve side, the characteristic polynomial of the Frobenius at $p$ is given by 
\begin{align*}
	(X-\alpha_p^2)(X-\alpha_p\beta_p)(X-\beta_p^2)\;=:\;X^3+c_{p,1}X^2+c_{p,2}X+c_{p,3},
\end{align*}
where $\alpha_p,\beta_p$ are two roots of $X^2-a_p(C)X+p$; on the modular form side, we have the following conjecture.

\begin{conjecture}
\label{k=4-asd-conjecture}
	Let $C/\Q$ be an elliptic curve with $j(C)\notin\{0,1728\}$ and let $p$ be a good prime of $C$ with  $v_p(j(C))=0=v_p(j(C)-1728)$. For $1\leq r\leq 3$, write $F^{(r)}=\frac{E_4}{(j-j(C))^r}$. Then, for all $n,l\in\Z^+$,
	\begin{align*}
		a_{np^l}(F^{(r)})+c_{p,1}\thin a_{np^{l-1}}(F^{(r)})+c_{p,2}\thin a_{np^{l-2}}(F^{(r)})+c_{p,3}\thin a_{np^{l-3}}(F^{(r)})\;\equiv\;0\pmod{p^{3l-2-r}}.
	\end{align*}
\end{conjecture}

Congruences of this form are usually called the \emph{Atkin--Swinnerton-Dyer (ASD) congruences}. This term originally refers to three-term congruences of the form
\begin{align*}
    a_{np}-A_p\thin a_n+\chi(p)p^{k-1}a_{n/p}\;\equiv\;0\pmod{p^{(k-1)(v_p(n)+1)}}
\end{align*}
for some sequence $(a_n)$. In general, it also refers to congruences of the form
\begin{align*}
    a_{np^l}+c_{p,1}\thin a_{np^{l-1}}+c_{p,2}\thin a_{np^{l-2}}+\cdots+c_{p,s}\thin a_{np^{l-s}}\;\equiv\;0\pmod{p^\bullet}
\end{align*}
for some suitable power of $p$. Atkin and Swinnerton-Dyer \cite{atkin-swinnerton-dyer} first came across this type of congruences when studying the coefficients of modular forms on non-congruence subgroups, whose result was later generalized by Scholl \cite{scholl-asd}. In \cite{kazalicki-scholl}, Kazalicki and Scholl demonstrated certain ASD congruences for weakly holomorphic modular forms, and as an example, showed that $F=E_4^6/\Delta-1464E_4^3$ satisfies that
\begin{align*}
    a_{np}(F)-\tau(p)a_n(F)+p^{11}a_{n/p}(F)\;\equiv\;0\pmod{p^{11v_p(n)}},
\end{align*}
where $\tau$ is the Ramanujan tau function. More recently, Allen, Long, and Saad \cite{allen-long-saad} established certain ASD congruences for meromorphic modular forms on finite index subgroups of~$\SL_2(\Z)$. For an overview of ASD congruences, we also refer to the paper by Li--Long \cite{li-long}.

The most interesting part of this story is perhaps the CM case, where the $p$-adic analogue of the Chowla--Selberg period occurs. The phenomena in this case are in a way governed by the two-variable modular form
\begin{align*}
	G_4(z,\tau)\deq\frac{E_4(z)\cdot\frac{E_{10}}{\Delta}(\tau)}{j(z)-j(\tau)}.
\end{align*}
Fix a fundamental discriminant $D<0$ with class number $1$. Let $w_D$ be the root number of $D$ and let $\alpha_D=\frac{D+\sqrt{D}}{2}\in\clh$. For $1\leq r\leq 3$, consider
\begin{align*}
	G_{4,D}^{(r)}(z)\deq (\partial_\tau^{r-1} G_4)(z,\alpha_D),
\end{align*}
where $\partial_\tau$ is the non-holomorphic derivative (with respect to $\tau$). Using CM theory, we will rescale each $G_{4,D}^{(r)}$ appropriately so that they have rational coefficients. Now, for a prime $p\nmid D$, define
\begin{align*}
	\Omega_{D,p}\deq -\big(\tfrac{D}{p}\big)\cdot\bigg(\prod_{a=1}^{|D|-1}\Gamma_p\bigg(\frac{a}{|D|}\bigg)^{(\frac{D}{a})}\thin\bigg)^{w_D/2}\;\in\;\Z_p^\times,
\end{align*}
where $\Gamma_p$ denotes the $p$-adic Gamma function. Then, we have the following conjectural supercongruences.

\begin{conjecture}
\zlabel{k=4-CM-supercongruence}
For $1\leq r\leq 3$, the modular form $G_{4,D}^{(r)}$ satisfies that
\begin{enumerate}[label=\rm{(\arabic*)}]
	\item if $\big(\frac{D}{p}\big)=1$, then for all $n,l\in\Z^+$,
	\begin{align*}
		a_{np^l}(G_{4,D}^{(r)})\;\equiv\; p^{r-1}\thin\Omega_{D,p}^{2-r}\thin a_{np^{l-1}}(G_{4,D}^{(r)})\pmod{p^{3l}};
	\end{align*}
	\item if $\big(\frac{D}{p}\big)=-1$, then for all $n,l\in\Z^+$,
	\begin{align*}
		a_{np^l}(G_{4,D}^{(4-r)})\;\equiv\;(-p)^{r-1}\thin\Omega_{D,p}^{-(2-r)}\thin a_{np^{l-1}}(G_{4,D}^{(r)})\pmod{p^{3l}}.
	\end{align*}
\end{enumerate}
\end{conjecture}

These supercongruences are provable when $r=2$. In addition, we establish the magnetic property of $G_{4,D}^{(2)}$. Specifically, we have the following result.

\begin{theorem}
Let $p\nmid D$ be a prime. Then, for all $n,l\in\Z^+$,
\begin{align*}
	a_{np^l}(G_{4,D}^{(2)})\;\equiv\;\big(\tfrac{D}{p}\big)p\thin  a_{np^{l-1}}(G_{4,D}^{(2)})\pmod{p^{3l}}.
\end{align*}
Moreover, up to a nonzero rational multiple, $G_{4,D}^{(2)}$ is $1$-magnetic.
\end{theorem}

The supercongruences and the magnetic property should be compared with the fact that the space $\Sym^2H_{\et}^1(C,\Q_\ell)$ splits when $C$ has CM, and in particular, that it splits into three $1$-dimensional pieces when viewed as a Galois representation of $\Q(\sqrt{D})$. 

We end the list of phenomena with an explicit connection between the $p$-th coefficient of $\frac{E_4}{j-c}$ and the truncated hypergeometric sum associated to the datum $\big((\frac{1}{2},\frac{1}{6},\frac{5}{6}),(1,1)\big)$
\begin{align*}
    \pFq{3}{2}{\frac{1}{2},\frac{1}{6},\frac{5}{6}}{1,1}{\frac{1728}{c}}_{p-1}\eq\sum_{m=0}^{p-1}\frac{(6m)!}{(m!)^3(3m)!}\thin c^{-m},
\end{align*}
where we adopt notation from \zcref{hypergeometric-subsection}. This connection also constitutes a key part of the proof of \zcref{k=4-good-prime-theorem}.

\begin{theorem}
\label{hypergeom-theorem-intro}
Let $c\in\Q$ and let $p\geq 5$ be a prime with $v_p(c)=0$. Then,
\begin{align*}
    a_p\bigg(\frac{E_4}{j-c}\bigg)&\;\equiv\;\big(c(c-1728)\big)^{\frac{p-1}{2}}\cdot\sum_{m=0}^{p-1}\frac{(6m)!}{(m!)^3(3m)!}\thin c^{-m}\pmod{p}.
\end{align*}
\end{theorem}

\subsection{Main results}
\label{main-result-section}

We now discuss the main results of this paper. The first result is a generalization of \zcref{k=4-good-prime-theorem} on meromorphic modular forms with simple poles. To describe the result, we introduce the notion of relations.

Let $k\in 2\Z_{\geq2}$. An integer sequence $\lambda=(\lambda_m)_{m\in\Z^+}$ is called a \emph{relation} for $S_k$ if
\begin{enumerate}[label=(\arabic*)]
	\item $\lambda_m=0$ for all but finitely many~$m$;
	\item $\sum_{m=1}^\infty \lambda_m a_m\eq 0$ for any cusp form $\sum_{m=1}^\infty a_mq^m\in S_k$.
\end{enumerate}
For a $q$-series $F$ and a relation $\lambda$ for $S_k$, write
\begin{align*}
	F|\lambda\deq\sum_{m=1}^\infty \lambda_m F|T_{m,k},
\end{align*}
where $T_{m,k}$ is the $m$-th Hecke operator of weight $k$.

\begin{theorem}
\label{good-prime-theorem-intro}
Let $k\in 2\Z_{\geq 2}$, $g$ be a holomorphic modular form of weight $k$ with integer coefficients, and $\lambda$ be a relation for $S_k$. Let $C/\Q$ be an elliptic curve with $j(C)\notin\{0,1728\}$ and $p\geq 5$ be a good prime of $C$ with $v_p(j(C))=0=v_p(j(C)-1728)$. Suppose that $\frac{g}{j-j(C)}$ has a pole. Then, for all $n\in\Z^+$,
\begin{align*}
    a_{np}\bigg(\frac{g}{j-j(C)}\bigg|\lambda\bigg)\;\equiv\;a_p(C)^{k-2}a_n\bigg(\frac{g}{j-j(C)}\bigg|\lambda\bigg)\pmod{p}.
\end{align*}
\end{theorem}
\begin{remark}
\label{good-prime-assumption-remark}
The theorem is conjecturally true for $p=2,3$. Depending on $k \pmod{12}$, the theorem is also conjecturally true for primes $p$ with $v_p(j(C))\neq 0$ or $v_p(j(C)-1728)\neq 0$ (see \zcref{simple-pole-cong-remark}).
\end{remark}

The second result is on the supercongruences and the magnetic property of certain meromorphic modular forms of weight~$k$ with CM poles of order $k/2$. In order to describe the result, we first introduce some notation.

Let $D<0$ be a discriminant, $\cl{O}_D$ be the imaginary quadratic order of discriminant $D$, and \mbox{$w_D=|\cl{O}_D^\times|$}. Let $\cl{Q}_D$ be a complete set of $\SL_2(\Z)$-inequivalent positive definite quadratic forms $Q(X,Y)\in\Z[X,Y]$ of discriminant $D$ and let $\cl{Q}_D^\prim\subseteq\cl{Q}_D$ be the subset consisting of primitive elements. For $Q(X,Y)=aX^2+bXY+cY^2\in\cl{Q}_D$, define
\begin{align*}
	\alpha_Q\deq\frac{-b+\sqrt{D}}{2a}\in\clh.
\end{align*}

For $k\in 2\Z$, define
\begin{align*}
    g_k\deq E_{k'}\Delta^\delta,
\end{align*}
where we write $k=k'+12\delta$ for some $k'\in\{0,4,6,8,10,14\}$ and $\delta\in\Z$. For a (nearly holomorphic) modular form $f(\tau)$ of weight $k_f$, we may view $f$ as a function in $\tau$ and $\ov{\tau}$ and define the \emph{non-holomorphic derivative} as
\begin{align*}
    (\partial_\tau f)(\tau)\deq \frac{1}{2\pi i}\frac{\partial f}{\partial \tau}-\frac{k_f}{4\pi \Im(\tau)}\cdot f(\tau).
\end{align*}
It is easy to check that $\partial_\tau f$ is modular of weight $k_f+2$. Thus, it is possible to apply the operator $\partial_\tau$ recursively to a modular form.

Now, for $k\in 2\Z_{\geq 2}$, let\footnote{This two-variable modular form $G_k(z,\tau)$ occurred in many places throughout the literature. We refer to \cite[p.~98]{asai-kaneko-ninomiya} for a detailed discussion and the references therein. It also plays a key role in Zagier's duality \cite{zagier-zagierduality,duke-jenkins-zagierduality}.}
\begin{align}
    G_k(z,\tau)\deq\frac{g_k(z)\cdot g_{2-k}(\tau)}{j(z)-j(\tau)}.
\end{align}
Define
\begin{align}
	\cl{G}_{k,D}(z)\deq\frac{2}{w_D}\sum_{Q\in\cl{Q}_D^\prim}(\partial^{\frac{k-2}{2}}_\tau G_k)(z,\alpha_Q).
\end{align}
Here $G_k(z,\tau)$ is modular of $2-k$ with respect to $\tau$, so $(\partial^{\frac{k-2}{2}}_\tau G_k)(z,\tau)$ is of weight $0$ with respect to $\tau$. Hence, $\cl{G}_{k,D}(z)$ is independent of the choice of~$\cl{Q}_D^\prim$.

\begin{theorem}
\zlabel{cm-middle-supercongruence-intro}
	Let $k\in 2\Z_{\geq2}$ and $\lambda$ be a relation for $S_k$. Let $D<0$ be a discriminant and write $D=A^2D_0$ for some $A\in\Z^+$ and some fundamental discriminant $D_0<0$. Let
	\begin{align}
		\wt{\cl{G}}_{k,D}\eq\begin{cases}
			|D_0|^{-\frac{k}{4}}\cl{G}_{k,D}&\text{ if }k\equiv0\pmod{4},\\
			|D_0|^{\frac{k-2}{4}}\cl{G}_{k,D}&\text{ if }k\equiv 2\pmod{4}.
		\end{cases}
	\end{align}
	Then, $\wt{\cl{G}}_{k,D}$ satisfies the following. 
	\begin{enumerate}[label=\rm{(\arabic*)}]
		\item $\wt{\cl{G}}_{k,D}\in\Z[[q]]$ and $\wt{\cl{G}}_{k,D}|\lambda$ is $\frac{k-2}{2}$-magnetic, i.e., for all $n\in\Z^+$,
		\begin{align*}
			n^{\frac{k-2}{2}}\thin\Big|\thin a_n\big(\wt{\cl{G}}_{k,D}|\lambda\big).
		\end{align*}
		\item Let $p$ be a prime with $p\nmid D$. Then, for all $n,l\in\Z^+$,
		\begin{align}
			a_{np^l}\big(\wt{\cl{G}}_{k,D}|\lambda\big)\;\equiv\;\Big(\big(\tfrac{D}{p}\big)p\Big)^{\frac{k-2}{2}}a_{np^{l-1}}\big(\wt{\cl{G}}_{k,D}|\lambda\big)\pmod{p^{(k-1)l}}.
		\end{align}
		\item Let $p$ be a prime with $p\divides D$ and $p\nmid A$. Then, for all $n,l\in\Z^+$,
		\begin{enumerate}[label=\rm{(\alph*)}]
			\item if $k\equiv 0\pmod{4}$, then
			\begin{align}
				a_{np^l}\big(\wt{\cl{G}}_{k,D}|\lambda\big)\;\equiv\;0\pmod{p^{(k-1)l}};
			\end{align}
			\item if $k\equiv 2\pmod{4}$, then
			\begin{align}
				a_{np^l}\big(\wt{\cl{G}}_{k,D}|\lambda\big)\;\equiv\;p^{\frac{k-2}{2}}a_{np^{l-1}}\big(\wt{\cl{G}}_{k,D}|\lambda\big)\pmod{p^{(k-1)l}}.
			\end{align}
		\end{enumerate}
	\end{enumerate}
\end{theorem}
\begin{remark}
\label{k=6-j=0-remark}
This theorem implies the supercongruences satisfied by $\frac{E_4}{j}$ and $\frac{E_4}{j-1728}$, namely \zcref{supercongruence-k=4-j=0-1728}, by choosing $k=4$ and $D=-3$ or $D=-4$ respectively. If one chooses $k=6$ and $D=-3$, then one obtains the following nice supercongruence for $\frac{E_6}{j}$: for all primes $p\geq 5$ and all $n,l\in\Z^+$,
\begin{align*}
    a_{np^l}\bigg(\frac{E_6}{j}\bigg)\;\equiv\;p^2\thin a_{np^{l-1}}\bigg(\frac{E_6}{j}\bigg)\pmod{p^{5l}}.
\end{align*}
In this case, this theorem also implies that $\frac{384E_6}{j}$ is $2$-magnetic, though the $2$-magnetic property of $\frac{E_6}{j}$ was already proved in \cite[Theorem~2]{pasol-zudilin}.
\end{remark}

\subsection{Structure of the paper}
\label{structure-section}

The first half of this paper will be dedicated to describing the $p$-adic properties of the meromorphic modular forms of our interests.

\zcref{general-section} will discuss the general case. We will start with the case of having a simple pole and raise the conjectural generalizations of \zcref{good-prime-theorem} through a supersingular-ordinary dichotomy. We will then move on to the case of having a pole of higher order. Specifically, we will focus on the following space of modular forms
\begin{align*}
    \MMF_{k,C}\deq\Span\bigg\{\frac{g_k}{j-j(C)},\;\frac{g_k}{(j-j(C))^2},\ldots,\frac{g_k}{(j-j(C))^{k-1}}\bigg\},
\end{align*}
or rather $\MMF_{k,C}|\lambda$ for any fixed relation $\lambda$ for $S_k$. We will make comparisons between the space $\MMF_{k,C}|\lambda$ with the $U_p$ action and the space $\Sym^{k-2}C$ with the Frobenius action, and as a generalization of \zcref{k=4-asd-conjecture}, we will conjecture a suitable ASD congruence satisfied by modular forms in $\MMF_{k,C}|\lambda$, which is given by the characteristic polynomial of the Frobenius action on $\Sym^{k-2}C$.

\zcref{cm-section} will focus on the CM case. For a fundamental discriminant $D<0$ of class number~$1$, we will specifically focus on the case of having a pole at a CM point $\alpha_D$ of discriminant~$D$, or equivalently, the space $\MMF_{k,D}:=\MMF_{k,C}$ when $C/\Q$ is a CM elliptic curve with CM by the imaginary quadratic order of discriminant $D$. We will construct a rational basis $\{G_{k,D}^{(r)}\}_{r=1}^{k-1}$ for the space $\MMF_{k,D}$, with each $G_{k,D}^{(r)}$ having a pole at $\alpha_D$ of order~$r$, such that, after applying a relation $\lambda$, the basis is conjecturally
\begin{enumerate}[label=(\arabic*)]
    \item diagonal under the action of (powers of) $U_p$ for ordinary primes $p$, and
    \item anti-diagonal under the action of (powers of) $U_p$ for supersingular primes $p$
\end{enumerate}
via explicit supercongruences, where the (anti-)eigenvalues are described explicitly using the $p$-adic analogue of the Chowla--Selberg periods. Moreover, up to some nonzero integer multiple, each $G_{k,D}^{(r)}|\lambda$ is conjecturally $(r'-1)$-magnetic, where $r'=\min(r,k-r)$. 

The second half of this paper will be devoted to the proofs of the main results. \zcref{notation-section} will recall the notation and preliminaries that are needed for the proofs. \zcref{good-prime-proof-section} will establish the hypergeometric connection, namely \zcref{hypergeom-theorem-intro}, and use this to prove a generalization of \zcref{good-prime-theorem-intro} for elliptic curves over number fields. \zcref{shimura-magnetic-section} will discuss the proof of \zcref{cm-middle-supercongruence-intro} using the Borcherds--Shimura lift.

Throughout the paper, we also include some numerical examples computed via PARI/GP \cite{pari}, sometimes aided by SageMath \cite{sagemath} and LMFDB \cite{lmfdb}.

\subsection{Related work}

One of the main motivations of this paper is to obtain an explicit (though possibly conjectural) connection between meromorphic modular forms and symmetric powers of elliptic curves through numerically verifiable congruences. As mentioned before, the connection itself was already discussed in Brown--Fonseca \cite{brown-fonseca} from the archimedean viewpoint, which built a comprehensive theory for meromorphic modular forms that is very closely related to this paper.

Indeed, in view of Kazalicki--Scholl \cite{kazalicki-scholl}, it is natural to expect the congruences that link meromorphic modular forms to symmetric powers of elliptic curves to be in the form of the ASD congruences. In this direction, the work of Allen--Long--Saad \cite{allen-long-saad} established certain ASD congruences for meromorphic modular forms on finite index subgroups of $\SL_2(\Z)$ by extending the methods of Scholl \cite{scholl-asd} and Kazalicki--Scholl \cite{kazalicki-scholl}. They also treated the CM case in detail and deduced the existence of the (anti-)diagonal basis, though without explicitly identifying the basis and the (anti-)eigenvalues. Indeed, their main results should partially imply some of the congruences conjectured in this paper subject to certain technical assumptions.

Another main issue that this paper encounters is the lack of Hecke operators on the space of meromorphic modular forms, or specifically, on the space $\MMF_{k,C}$. A reasonable workaround that one may expect is via the theory of $p$-adic modular forms. From this viewpoint, the work of Bordignon \cite{bordignon} related meromorphic modular forms to overconvergent $p$-adic modular forms and described the explicit action of the $U_p$ operator via comparison theorems between rigid and de Rham cohomology, making it possible to deduce certain congruences of meromorphic modular forms by using the integrality arising from crystalline cohomology, and hence potentially resolving some of the conjectures in this paper subject to certain technical assumptions.

\subsection*{Acknowledgements}

The author would like to thank many people for useful discussions on this project, including Kilian B\"onisch, Steven Charlton, Chi Hong Chow, Henri Cohen, H\r{a}vard Damm-Johnsen, Vasily Golyshev, Wen-Ching Winnie Li, Yingkun Li, Ling Long, Alexandre Maksoud, Danylo Radchenko, Berend Ringeling, Fernando Rodriguez Villegas, Hasan Saad, and Jan Vonk. The author is in particular grateful to Don Zagier and Wadim Zudilin for all the useful ideas on proofs and generalizations.

This project originated from discussions during the International Conference on
Modular Forms and $q$-Series at University of Cologne, and during the follow-up workshop at the Hausdorff Research Institute for Mathematics to the trimester program `Periods in Number Theory, Algebraic Geometry and Physics' funded by the Deutsche Forschungsgemeinschaft (DFG, German Research Foundation) under Germany's Excellence Strategy – EXC-2047/1 – 390685813. Part of the later discussions also took place during the author's visit to the Abdus Salam International Centre for Theoretical Physics hosted by Vasily Golyshev and Don Zagier, and during the author's visit to Louisiana State University hosted by Ling Long and partially supported by the Simons Foundation grant MP-TSM-00002492.

\section{The setup}
\label{setup-section}

We first introduce our setup. Throughout this paper, the main objects are meromorphic modular forms of level $1$ and weight $k\geq 4$ with exactly one pole at a non-cuspidal point. Specifically, we will investigate the $p$-adic properties of modular forms of the form
\begin{align*}
	\frac{g}{(j-c)^r}
\end{align*}
and their linear combinations, where $g$ is a holomorphic modular form of level $1$ with integer coefficients and $c\in\Q$. We will also fix an elliptic curve $C/\Q$ with $j(C)=c$. In almost all the cases, we will assume that $j(C)\notin\{0,1728\}$ and that $v_p(j(C))=0=v_p(j(C)-1728)$ for the prime $p$ unless stated otherwise.

\subsection{The cases when $k\in\{4,6,8,10,14\}$}

When $k\in\{4,6,8,10,14\}$, $g=E_k$ is essentially the only possible choice. We will thus focus on modular forms of the form
\begin{align}
	F_{k,C}^{(r)}\deq\frac{E_k}{(j-j(C))^r}.
\end{align}
In fact, we will only focus on the case when $1\leq r\leq k-1$ for the following reasons.
\begin{enumerate}[label=\textbullet,leftmargin=20pt]
	\item It is easy to check that each $F_{k,C}^{(r)}$ can be written as a linear combination of $F_{k,C}^{(1)},\ldots, F_{k,C}^{(k-1)}$, and $\D^{k-1}f$ for some modular form $f$ of weight $2-k$, where $\D=q\frac{\dif}{\dif q}$. As the properties of $\D^{k-1}f$ are easy to understand,\footnote{The modular form $\D^{k-1}f$ is also trivially $(k-1)$-magnetic, and hence not interesting either from the viewpoint of the magnetic property.} it suffices to consider the denominator only up to the $(k-1)$-st power.
	\item The discussion above yields a non-canonical isomorphism
	\begin{align}
		\MMF_{k,C}\deq\Span\bigg\{\frac{E_k}{j-j(C)},\;\frac{E_k}{(j-j(C))^2},\ldots,\frac{E_k}{(j-j(C))^{k-1}}\bigg\}\;\cong\;\frac{S_k^{\mero,C}}{\D^{k-1}M_{2-k}^{\mero,C}}
	\end{align}
	for $k\in\{4,6,8,10,14\}$ and $j(C)\notin\{0,1728\}$, where $M_l^{\mero,C}$ denotes the space of meromorphic modular forms of weight $l$ and level $1$ with precisely one pole at (the point on $\SL_2(\Z)\backslash\clh$ corresponding to) $C$, and $S_l^{\mero,C}$ denotes its subspace consisting of elements with vanishing constant terms in the Fourier expansions. Here Brown and Fonseca \cite{brown-fonseca} provided a cohomological interpretation of the latter quotient and relate it to $\Sym^{k-2}H^1_{\dr}(C)$.
\end{enumerate}
In the later discussion, we will present explicit congruences for elements in the space $\MMF_{k,C}$, aided by numerical examples, which then visualizes its clear connection to $\Sym^{k-2}C$.

\subsection{The general case}

More generally, one can choose any holomorphic modular form $g$ of level $1$ and weight $k$ with integer coefficients and form a modular form
\begin{align*}
	\frac{g}{(j-j(C))^r}.
\end{align*}
The behaviors of these meromorphic modular forms are then an interplay of $\Sym^{k-2}C$ and cusp forms of weight $k$.

To mimic the cases when $k\in\{4,6,8,10,14\}$, we will remove the influence of cusp forms using suitable linear combinations of Hecke operators. Recall that an integer sequence $\lambda=(\lambda_m)_{m\in\Z^+}$ is called a \emph{relation}\footnote{The consideration of relations first appeared in Gross--Zagier \cite[p.~316]{gross-zagier}. This idea was used in L\"{o}brich--Schwagenscheidt \cite[p.~2383]{loebrich-schwagenscheidt} to construct certain magnetic modular forms and in Brown--Fonseca \cite[Remark~8.10]{brown-fonseca} to study meromorphic modular forms.} for $S_k$ if
\begin{enumerate}[label=(\arabic*)]
	\item $\lambda_m=0$ for all but finitely many~$m$;
	\item $\sum_{m=1}^\infty \lambda_m a_m\eq 0$ for any cusp form $\sum_{m=1}^\infty a_mq^m\in S_k$.
\end{enumerate}
For a $q$-series $F$ and a relation $\lambda$ for $S_k$, write
\begin{align}
	F|\lambda\deq\sum_{m=1}^\infty \lambda_m F|T_{m,k},
\end{align}
where $T_{m,k}$ is the $m$-th Hecke operator of weight $k$. It follows from the definition that $f|\lambda=0$ for any relation $\lambda$ for $S_k$ and any cusp form $f\in S_k$. Hence, replacing $F$ by $F|\lambda$ will `kill the cuspidal part' of~$F$. In this way, we can fix a relation $\lambda$ for $S_k$ and replace $\frac{g}{(j-j(C))^r}$ by $\frac{g}{(j-j(C))^r}|\lambda$, so that the resulting modular form is again solely governed by the corresponding elliptic curve.

In fact, as in the cases when $k\in\{4,6,8,10,14\}$ where it suffices to consider only $g=E_k$, for a general weight~$k$, it is also sufficient to consider one particular holomorphic modular form. Recall that for $k\in 2\Z$, define
\begin{align}
\label{g-k-def}
	g_k\deq E_{k'}\Delta^\delta,
\end{align}
where we write $k=k'+12\delta$ for some $k'\in\{0,4,6,8,10,14\}$ and $\delta\in\Z$. Indeed, $g_k=E_k$ when $k\in\{4,6,8,10,14\}$.

Before stating the next lemma, we recall that for a prime $p$, one defines
\begin{align*}
	\Z_{(p)}\deq\{a\in\Q\mid v_p(a)\geq0\}.
\end{align*}

\begin{lemma}
	\zlabel{one-choice-for-k}
	Let $k\in 2\Z_{\geq2}$. Let $p$ be a prime and let $\alpha\in\clh$ satisfy that $j(\alpha)\in\Z_{(p)}$. 
	Let $g\in\Z_{(p)}[[q]]$ be a holomorphic modular form of level $1$ and weight $k$. Then, for any $r\in\Z^+$, $\frac{g}{(j-j(\alpha))^r}$ is a $\Z_{(p)}$-linear combination of $\frac{g_k}{j-j(\alpha)},\frac{g_k}{(j-j(\alpha))^2},\ldots,\frac{g_k}{(j-j(\alpha))^r}$, and a cusp form of weight~$k$ and level $1$ with $\Z_{(p)}$-coefficients.
\end{lemma}
\begin{proof}
	We will prove this by induction. By analyzing the location of the poles, it is easy to check that $g/g_k$ can be written as a polynomial in $j$, which has $p$-integral coefficients (see also \cite[(2)]{kaneko-zagier}). In particular, $A:=g(\alpha)/g_k(\alpha)\in\Z_{(p)}$. Write 
	\begin{align*}
		h\eq\frac{g-A g_k}{j-j(\alpha)}\in\Z_{(p)}[[q]].
	\end{align*}
	Then, $h$ is holomorphic on $\clh$ and vanishes at $\infty$, so $h\in S_k$. Now, for $r=1$,
	\begin{align*}
		\frac{g}{j-j(\alpha)}\eq h+A\cdot\frac{g_k}{j-j(\alpha)},
	\end{align*}
	so it satisfies the requirement. For $r>1$,
	\begin{align*}
		\frac{g}{(j-j(\alpha))^r}\eq\frac{h}{(j-j(\alpha))^{r-1}}+A\cdot\frac{g_k}{(j-j(\alpha))^r}.
	\end{align*}
	The result then follows by applying the induction hypothesis to $h$ and $r-1$.
\end{proof}

Since replacing $\frac{g}{(j-j(C))^r}$ by $\frac{g}{(j-j(C))^r}|\lambda$ kills the cuspidal part, \zcref{one-choice-for-k} implies that it is sufficient to study the $p$-adic properties of $\frac{g_k}{(j-j(C))^r}|\lambda$. We will thus focus on
\begin{align}
	F_{k,C}^{(r)}\deq \frac{g_k}{(j-j(C))^r}
\end{align}
and their linear combinations for $k\in 2\Z_{\geq2}$. For essentially the same reasons as before, we will only consider the case when $1\leq r\leq k-1$, namely we will concentrate on the space
\begin{align}
	\MMF_{k,C}\deq\Span\bigg\{\frac{g_k}{j-j(C)},\;\frac{g_k}{(j-j(C))^2},\ldots,\frac{g_k}{(j-j(C))^{k-1}}\bigg\}.
\end{align}

\section{General meromorphic modular forms}
\label{general-section}

\subsection{Simple poles}
\label{simple-pole-section}

We first start with the simplest case, i.e., the case of meromorphic modular forms with precisely one pole that is simple and at a non-cuspidal point. In this case, the main focus is the modular form
\begin{align}
	F_{k,C}^{(1)}\eq\frac{g_k}{j-j(C)}
\end{align}
with $j(C)\notin\{0,1728\}$. Here, one can already observe the following congruence, which is one of the main results of this paper.

\begin{theorem}
\zlabel{good-prime-theorem}
Let $k\in 2\Z_{\geq2}$, $C/\Q$ be an elliptic curve with $j(C)\notin\{0,1728\}$, $\lambda$ be a relation for $S_k$, and $p\geq 5$ be a good prime of $C$ with $v_p(j(C))=0=v_p(j(C)-1728)$. Then, for all $n\in\Z^+$,
\begin{align}
	a_{np}\big(F_{k,C}^{(1)}|\lambda\big)\;\equiv\;a_p(C)^{k-2}a_n\big(F_{k,C}^{(1)}|\lambda\big)\pmod{p}.
\end{align}
\end{theorem}
\begin{remark}
\label{simple-pole-cong-remark}
This congruence is conjecturally true for $p=2,3$ as well. Moreover, depending on $k \pmod{12}$, this congruence is also conjecturally true for primes $p$ with $v_p(j(C))\neq 0$ or $v_p(j(C)-1728)\neq 0$, which we describe here. Let
	\begin{align*}
		S_{0,C}\deq\{p\mid v_p(j(C))>0\}\quad\text{ and }\quad S_{1728,C}\deq\{p\mid v_p(j(C)-1728)>0\}.
	\end{align*}
	For $k\in 2\Z_{\geq2}$, define
	\begin{align*}
		S^{(k,C)}\deq\begin{cases}
			S_{0,C}\cup S_{1728,C}&\text{ if }k\equiv 4,12\pmod{12} \\
			S_{0,C}&\text{ if }k\equiv 6,10\pmod{12} \\
			S_{1728,C}&\text{ if }k\equiv 8\pmod{12} \\
			\emptyset&\text{ if }k\equiv 14\pmod{12}.
		\end{cases}
	\end{align*}
	Then, the congruence should hold conjecturally for all good primes $p$ of $C$ with $p\notin S^{(k,C)}$. For example, when $k\equiv 14\pmod{12}$, the congruence should simply hold conjecturally for all good primes of $C$.
\end{remark}
\begin{remark}
From a different point of view, this describes the behavior of $F_{k,C}^{(1)}|\lambda$ under the $U_p$ operator. In particular, it says that $F_{k,C}^{(1)}|\lambda$, when viewed as a mod $p$ modular form, is a $U_p$-eigenform with eigenvalue $a_p(C)^{k-2}$.
\end{remark}

\zcref{good-prime-theorem} has a conjectural generalization under an ordinary-supersingular dichotomy. Recall that a good prime $p$ of $C$ is called \emph{ordinary} if $p\nmid a_p(C)$ and \emph{supersingular} if $p\divides a_p(C)$. We have the following conjectures.

\begin{conjecture}
	\zlabel{ordinary-prime-conjecture}
	Let $k\in 2\Z_{\geq 2}$, $C/\Q$ be an elliptic curve with $j(C)\notin\{0,1728\}$, $\lambda$ be a relation for $S_k$, and $p$ be an ordinary prime of $C$ with $v_p(j(C))=0=v_p(j(C)-1728)$. Then, for all $n,l\in\Z^+$,
	\begin{align}
		a_{np^l}\big(F_{k,C}^{(1)}|\lambda\big)\;\equiv\;u_p(C)^{k-2}a_{np^{l-1}}\big(F_{k,C}^{(1)}|\lambda\big)\pmod{p^l},
	\end{align}
	where $u_p(C)$ is the root of $X^2-a_p(C)X+p$ which is a $p$-adic unit.
	
	Moreover, if $C/\Q$ has CM, then for all $n,l\in\Z^+$,
	\begin{align*}
		a_{np^l}\big(F_{k,C}^{(1)}|\lambda\big)\;\equiv\;u_p(C)^{k-2}a_{np^{l-1}}\big(F_{k,C}^{(1)}|\lambda\big)\pmod{p^{(k-1)l}}.
	\end{align*}
\end{conjecture}

\begin{conjecture}
	\zlabel{supersingular-prime-conjecture-simple}
	Let $k\in 2\Z_{\geq 2}$, $C/\Q$ be an elliptic curve with $j(C)\notin\{0,1728\}$, $\lambda$ be a relation for $S_k$, and $p$ be a supersingular prime of $C$ with $p\geq k-1$ and $v_p(j(C))=0=v_p(j(C)-1728)$. Then, for all $n,l\in\Z^+$,
	\begin{align*}
		a_{np^l}\big(F_{k,C}^{(1)}|\lambda\big)\;\equiv\;p^{k-2}a_{np^{l-2}}\big(F_{k,C}^{(1)}|\lambda\big)\pmod{p^{(k-1)l-1}},
	\end{align*}
	where $a_m(\;\cdot\;):=0$ if $m\notin\Z$.
\end{conjecture}

In the later discussion, we will generalize \zcref{supersingular-prime-conjecture-simple} for $F_{k,C}^{(r)}|\lambda$ when $r\geq 2$ and investigate primes $p$ with $p\leq k-2$ as well. In contrast, a generalization of \zcref{ordinary-prime-conjecture} seems to exist only in the CM case, which will be discussed in \zcref{cm-section}. Here we finish this subsection with an example when $k=12$.

\begin{example}
\zlabel{k=12-simple-pole-example}
Let $k=12$ and $C/\Q$ be an elliptic curve with $j(C)\notin\{0,1728\}$. Then,
\begin{align*}
	F_{12,C}^{(1)}\eq\frac{\Delta}{j-j(C)}.
\end{align*}
Since
	\begin{align*}
		\Delta\eq\sum_{m=1}^\infty \tau(m)q^m
	\end{align*}
	is the only cusp form of level $1$ and weight $12$, for every $m\geq 2$, there is an obvious choice of the relation $\lambda_m=(-\tau(m),0,\ldots,0,1,0,0,\ldots)$, where the $1$ is at the $m$-th entry. Then,
	\begin{align*}
		F_{12,C}^{(1)}|\lambda_m\eq F_{12,C}^{(1)}|_{12}\big(T_{m,12}-\tau(m)\big).
	\end{align*}
	
	Let $p\geq 5$ be a good prime of $C$ with $v_p(j(C))=0=v_p(j(C)-1728)$. Then, applying \zcref{good-prime-theorem} with $n=1$ yields that
	\begin{align*}
		a_p\big(F_{12,C}^{(1)}|\lambda_m\big)\;\equiv\; a_p(C)^{10}\thin a_1\big(F_{12,C}^{(1)}|\lambda_m\big)\pmod{p},
	\end{align*}
	or equivalently,
	\begin{align*}
		a_{mp}(F_{12,C}^{(1)})-\tau(m)\thin a_p(F_{12,C}^{(1)})\;\equiv\;a_m(F_{12,C}^{(1)})\thin a_p(C)^{10}-\tau(m)\thin a_p(C)^{10}\pmod{p}.
	\end{align*}
	
	Let $p$ be a supersingular prime of $C$ with $v_p(j(C))=0=v_p(j(C)-1728)$ (and $p\geq 11$). Then, one observes numerically that,
	\begin{align*}
		a_p\big(F_{12,C}^{(1)}|\lambda_m\big)\;\equiv\; 0\pmod{p^{10}},
	\end{align*}
	or equivalently,
	\begin{align*}
		a_{mp}(F_{12,C}^{(1)})-\tau(m)\thin a_p(F_{12,C}^{(1)})\;\equiv\;0\pmod{p^{10}}.
	\end{align*}
	
	This example essentially shows that the coefficients of $F_{12,C}^{(1)}$ are conjecturally `determined' (up to congruences) by the coefficients of $\Delta$ and the $a_p(C)$'s.
\end{example}

\subsection{Poles of higher order}
\label{higher-pole-section}

We will now turn to the case of meromorphic modular forms with higher order poles. Specifically, we will discuss the Atkin--Swinnerton-Dyer congruences satisfied by modular forms of the form
\begin{align*}
	F_{k,C}^{(r)}\eq\frac{g_k}{(j-j(C))^r}
\end{align*}
for $1\leq r\leq k-1$ and $j(C)\notin\{0,1728\}$. In the supersingular case, we will also draw the connection to overconvergent modular forms.

In fact, when investigating the $p$-adic properties of $F_{k,C}^{(r)}$ at supersingular primes, it turns out that
\begin{align}
	\wt{F}_{k,C}^{(r)}\deq(k-1-r)!\cdot F_{k,C}^{(r)}
\end{align}
is a better normalization. We will thus describe all the properties mainly in terms of $\wt{F}_{k,C}^{(r)}$. It is then useful to define
\begin{align}
	\delta_{p,r}\eq\delta(k,C,p,r)\deq\begin{cases}
		v_p((k-1-r)!) & \text{if $p$ is ordinary,} \\
		0 & \text{if $p$ is supersingular.}
	\end{cases}
\end{align}

As suggested constantly, meromorphic modular forms of weight $k$ with precisely one pole at a non-cuspidal point should behave similarly to the symmetric $(k-2)$-nd power of the elliptic curve corresponding to the pole. The discrepancy, however, is that modular forms typically correspond to $2$-dimensional objects from representation theory, while symmetric $(k-2)$-nd powers of elliptic curves are $(k-1)$-dimensional. As described in \zcref{setup-section}, the `correct' space of meromorphic modular forms to consider should be
\begin{align}
	\MMF_{k,C}\deq\Span\bigg\{\frac{g_k}{j-j(C)},\;\frac{g_k}{(j-j(C))^2},\ldots,\frac{g_k}{(j-j(C))^{k-1}}\bigg\}
\end{align}
(for $j(C)\notin\{0,1728\}$), or rather, the space $\MMF_{k,C}|\lambda$ for any fixed relation $\lambda$ for $S_k$. This is a $(k-1)$-dimensional space and should `correspond' to the $(k-1)$-dimensional space $\Sym^{k-2}C$ on the elliptic curve side.

As discussed in \zcref{setup-section}, an important reason for considering this space $\MMF_{k,C}$ is its relation to $\Sym^{k-2}H^1_{\dr}(C)$ via the cohomological interpretation by Brown--Fonseca \cite{brown-fonseca}. Indeed, Kazalicki and Scholl \cite{kazalicki-scholl} constructed a similar space for weakly holomorphic modular forms and established the relevant ASD congruences using the cohomological interpretation. This thus suggests that elements in $\MMF_{k,C}$ should be linked to symmetric powers of elliptic curves via ASD congruences.

The guiding philosophy in our case is that the action of $U_p$ on the modular form side\footnote{Indeed, there is no $U_p$ action on the space $\MMF_{k,C}$, since the $U_p$ operator increases the level and shifts the pole. Still, we will formally consider the $U_p$ operator as defined by $\sum a_nq^n\mapsto \sum a_{pn}q^n$.} should correspond to the action of the Frobenius at $p$ on the elliptic curve side. 
\begin{center}
	\begin{tikzcd}[row sep=30pt, nodes={inner sep=7pt}]
		\MMF_{k,C}|\lambda \arrow[r, leftrightarrow, dashed] \arrow["U_p"'{name=A}, loop, distance=2em, in=295, out=245] & [1.5em] \Sym^{k-2}C \arrow["\Frob_p"'{name=C}, loop, distance=2em, in=295, out=245] \arrow[leftrightarrow, dashed, bend right=10, from=A, to=C]
	\end{tikzcd}
\end{center}
Let $p$ be a good prime of $C$ and let $P_{k,p}(X)$ denote the characteristic polynomial of the Frobenius at $p$ on $\Sym^{k-2}H_\et^1(C,\Q_\ell)$. Then,
\begin{equation}
\begin{aligned}
	P_{k,p}(X)&\eq(X-\alpha_p^{k-2})(X-\alpha_p^{k-3}\beta_p)\cdots(X-\alpha_p\beta_p^{k-3})(X-\beta_p^{k-2}) \\
	&\;=:\; \sum_{i=0}^{k-1}c_{p,k-1-i}^{(k)}X^i\eq X^{k-1}+\sum_{i=0}^{k-2}c_{p,k-1-i}^{(k)}X^i,
\end{aligned}
\end{equation}
where $\alpha_p$ and $\beta_p$ are the two roots of $X^2-a_p(C)X+p$. We then have the following conjecture.

\begin{conjecture}
	\zlabel{asd-conjecture}
	Let $k\in 2\Z_{\geq2}$, $1\leq r\leq k-1$, $C/\Q$ be an elliptic curve with $j(C)\notin\{0,1728\}$, and $\lambda$ be a relation for $S_k$. Let $p$ be a good prime of $C$ with $v_p(j(C))=0=v_p(j(C)-1728)$. Then, for all $n,l\in\Z^+$,
	\begin{equation}
	\begin{aligned}
		a_{np^l}\big(\wt{F}_{k,C}^{(r)}|\lambda\big)+c_{p,1}^{(k)}\thin a_{np^{l-1}}\big(\wt{F}_{k,C}^{(r)}|\lambda\big)+\cdots+&c_{p,k-1}^{(k)}\thin a_{np^{l-k+1}}\big(\wt{F}_{k,C}^{(r)}|\lambda\big)\;\equiv\;0 \\
		&\hspace{2.5cm}\pmod{p^{(k-1)l-\frac{1}{2}(k-3)k-r+\delta_{p,r}}},
	\end{aligned}
	\end{equation}
	where $a_m(\;\cdot\;)=0$ if $m\notin\Z$.
\end{conjecture}

\begin{remark}
	If one simply considers the cases when $l\geq k-1$ and writes $s=l-k+1$, then \zcref{asd-conjecture} can be rewritten as the following: for all $s\in\Z_{\geq0}$,
	\begin{align*}
		U_p^s \thin P_{k,p}(U_p)\big(\wt{F}_{k,C}^{(r)}|\lambda\big)\;\equiv\;0\pmod{p^{(k-1)s+\frac{1}{2}(k^2-k+2)-r+\delta_{p,r}}}.
	\end{align*}
	On the other hand, one may also rewrite \zcref{asd-conjecture} using $V_p$ as for all $l\in\Z^+$,
	\begin{align*}
		U_p^l \thin P_{k,p}^*(V_p)\big(\wt{F}_{k,C}^{(r)}|\lambda\big)\;\equiv\;0\pmod{p^{(k-1)l-\frac{1}{2}(k-3)k-r+\delta_{p,r}}},
	\end{align*}
	where $P_{k,p}^*$ is the reciprocal polynomial of $P_{k,p}$, i.e., $P_{k,p}^*(X)=X^{k-1}P_{k,p}(X^{-1})$.\footnote{In view of Kazalicki--Scholl \cite{kazalicki-scholl}, the formulation using $V_p$ may be the `correct' one, since their proof of the ASD congruences involves identifying a $V_p$-related operator with the Frobenius at $p$.}
\end{remark}

\begin{remark}
\zlabel{nontilde-normalization}
The different normalizations $F_{k,C}^{(r)}$ and $\wt{F}_{k,C}^{(r)}$  only affect supersingular primes $p\leq k-2$. When $k\in\{4,6,8,10,14\}$, any supersingular prime $p\leq k-2$ satisfies that $j(C)\equiv0,1728\pmod{p}$, which is then excluded by the assumption that $v_p(j(C))=0=v_p(j(C)-1728)$. Hence, for $k\in\{4,6,8,10,14\}$, one can simply rewrite \zcref{asd-conjecture} as
\begin{equation}
\begin{aligned}
	a_{np^l}\big(F_{k,C}^{(r)}|\lambda\big)+c_{p,1}^{(k)}\thin a_{np^{l-1}}\big(F_{k,C}^{(r)}|\lambda\big)+\cdots+&c_{p,k-1}^{(k)}\thin a_{np^{l-k+1}}\big(F_{k,C}^{(r)}|\lambda\big)\;\equiv\;0 \\
	&\hspace{2.5cm}\pmod{p^{(k-1)l-\frac{1}{2}(k-3)k-r}}.
\end{aligned}
\end{equation}
\end{remark}

\begin{remark}
	\zlabel{k=6-stronger-asd}
	Note that $(k-1)l-\frac{1}{2}(k-3)k-r<0$ when $l=1$ and $r=1$. It is thus not possible to recover \zcref{good-prime-theorem} from \zcref{asd-conjecture} by setting $l=1$ and $r=1$. This is not the only case where the actual congruences appear stronger than the ones predicted by \zcref{asd-conjecture}. For example, when $k=6$, it appears that the following stronger congruences are true: for all $n\in\Z^+$ and $1\leq r\leq 5$,
	\begin{align*}
		a_{np^2}(F_{6,C}^{(r)})+c_{p,1}^{(6)}\thin a_{np}(F_{6,C}^{(r)})+c_{p,2}^{(6)}\thin a_n(F_{6,C}^{(r)})&\;\equiv\;0\pmod{p^{4-r}} \\
		a_{np^3}(F_{6,C}^{(r)})+c_{p,1}^{(6)}\thin a_{np^2}(F_{6,C}^{(r)})+c_{p,2}^{(6)}\thin a_{np}(F_{6,C}^{(r)})+c_{p,3}^{(6)}\thin a_n(F_{6,C}^{(r)})&\;\equiv\;0\pmod{p^{7-r}}.
	\end{align*}
\end{remark}

We give an example of \zcref{asd-conjecture} when $k=6$ and $r=2$, which also illustrates the stronger congruences described in \zcref{k=6-stronger-asd}. By \zcref{nontilde-normalization}, we will simply adopt the normalization~$F_{k,C}^{(r)}$.

\begin{example}
	Let $k=6$, $r=2$, and $C/\Q$ be the elliptic curve given by $y^2+y=x^3-x^2$ with LMFDB label \href{https://www.lmfdb.org/EllipticCurve/Q/11/a/3}{11.a3} and $j(C)=-4096/11$. We will look at
	\begin{align*}
		F\eq F_{6,C}^{(2)}\eq\frac{E_6}{(j+4096/11)^2}.
	\end{align*}
	In this case,
	\begin{align*}
		(k-1)l-\tfrac{1}{2}(k-3)k-r\eq 5l-11.
	\end{align*}
	
	Now, let $p=5$, which is an ordinary prime of $C$ with $v_p(j(C))=0=v_p(j(C)-1728)$ and $a_5(C)=1$. The characteristic polynomial of the Frobenius at~$5$ on $\Sym^4C$ is
	\begin{align*}
		P_{6,5}(X)\eq X^5-11X^4-495X^3+12375X^2+171875X-5^{10}.
	\end{align*}
	For this example, we computed the first $30000$ coefficients of $F$ and observed the following numerically for suitable $n\in\Z^+$
	{\small
		\begin{alignat*}{2}
			0&\thin\equiv\thin a_{5^2n}(F)-11\thin a_{5n}(F)-495\thin a_n(F) &&\pmod{5^2} \\
			0&\thin\equiv\thin a_{5^3n}(F)-11\thin a_{5^2n}(F)-495\thin a_{5n}(F)+12375\thin a_n(F) &&\pmod{5^5} \\
			0&\thin\equiv\thin a_{5^4n}(F)-11\thin a_{5^3n}(F)-495\thin a_{5^2n}(F)+12375\thin a_{5n}(F)+171875\thin a_n(F) &&\pmod{5^9} \\
			0&\thin\equiv\thin a_{5^5n}(F)-11\thin a_{5^4n}(F)-495\thin a_{5^3n}(F)+12375\thin a_{5^2n}(F)+171875\thin a_{5n}(F)-5^{10}\cdot a_n(F) &&\pmod{5^{15}}.
	\end{alignat*}}%
	Here the congruences when $l=2,3$ are stronger than the ones predicted by \zcref{asd-conjecture}, as discussed in \zcref{k=6-stronger-asd}; the one when $l=4$ is the same as predicted by \zcref{asd-conjecture}; the one when $l=5$ is stronger again, but only for $p=5$.
\end{example}

We will now specialize at supersingular primes. In this case, in addition to having the full ASD congruence, one also observes a generalization of the two-term congruence in \zcref{supersingular-prime-conjecture-simple}.

\begin{conjecture}
	\zlabel{supersingular-prime-conjecture}
	Let $k\in 2\Z_{\geq2}$, $1\leq r\leq k-1$, $C/\Q$ be an elliptic curve with $j(C)\notin\{0,1728\}$, and $\lambda$ be a relation for $S_k$. Let $p$ be a supersingular prime of $C$ with $v_p(j(C))=0=v_p(j(C)-1728)$. Then, for all $n,l\in\Z^+$,
	\begin{align}
		a_{np^l}\big(\wt{F}_{k,C}^{(r)}|\lambda\big)\;\equiv\;p^{k-2}a_{np^{l-2}}\big(\wt{F}_{k,C}^{(r)}|\lambda\big)\pmod{p^{(k-1)l-r}},
	\end{align}
	where $a_m(\;\cdot\;):=0$ if $m\notin\Z$.
\end{conjecture}

\begin{remark}
	This two-term congruence implies the full ASD congruence (\zcref{asd-conjecture}) in the supersingular case. We sketch a proof here for $l\geq k-1$. For a supersingular prime $p$, the characteristic polynomial $P_{k,p}(X)$ is given by
	\begin{align*}
		P_{k,p}(X)\eq&\prod_{i=0}^{k-2}\big(X-(\sqrt{-p})^{k-2-i}(-\sqrt{-p})^i\big) \\
		\eq &(X-(-p)^{\frac{k-2}{2}})(X^2-p^{k-2})^{\frac{k-2}{2}} \\
		\eq &\bigg(\sum_{i=0}^{\frac{k-4}{2}}\binom{\frac{k-4}{2}}{i}(-p^{k-2})^iX^{k-3-2i}-(-p)^{\frac{k-2}{2}}\sum_{i=0}^{\frac{k-4}{2}}\binom{\frac{k-4}{2}}{i}(-p^{k-2})^iX^{k-4-2i}\bigg)(X^2-p^{k-2}).
	\end{align*}
	It then follows from \zcref{supersingular-prime-conjecture} and direct computations that for all $s\in\Z_{\geq0}$, 
	\begin{align*}
		U_p^s\thin P_{k,p}(U_p)\big(\wt{F}_{k,C}^{(r)}|\lambda\big)
		\;\equiv\;0\pmod{p^{(k-1)s+\frac{1}{2}(k^2-k+2)-r}},
	\end{align*}
	which implies \zcref{asd-conjecture}.
\end{remark}

In fact, one can explore a different heuristic in this supersingular case. When $p$ is supersingular, $\wt{F}_{k,C}^{(r)}$ may be viewed as an overconvergent $p$-adic modular form, since the pole lies in the supersingular locus. Now, for $l\geq 2$, \zcref{supersingular-prime-conjecture} leads to the consideration of $U_p^{l-2}(U_p-p^{k-2})(\wt{F}_{k,C}^{(r)}|\lambda)$. Suppose that we would like to write $(U_p-p^{k-2})(\wt{F}_{k,C}^{(r)}|\lambda)$ as an asymptotic sum of $U_p$-eigenforms.\footnote{We refer to \cite[Section~3.5]{calegari-pmf} and \cite{gouvea-mazur} for discussions on the asymptotic expansions of overconvergent modular forms in terms of $U_p$-eigenforms.} One observes that the relation $\lambda$ annihilates all the cusp forms of level $1$ and weight $k$ and the operator $U_p-p^{k-2}$ annihilates all the newforms of level~$p$ and weight~$k$. Since Coleman's classicality theorem \cite{coleman-classicality} shows that any overconvergent $U_p$-eigenforms with slope $<k-1$ are classical, we should thus expect that
\begin{align*}
	(U_p-p^{k-2})\big(\wt{F}_{k,C}^{(r)}|\lambda\big)\;\sim\;\sum_i A_if_i
\end{align*}
where $A_i\in\ov{\Q_p}$ and each $f_i$ is an overconvergent $U_p$-eigenform with slope $\geq k-1$. In this way,
\begin{align*}
	U_p^{l-2}(U_p-p^{k-2})\big(\wt{F}_{k,C}^{(r)}|\lambda\big)\;\sim\;\sum_i A_i U_p^{l-2}f_i,
\end{align*}
and the latter has slope $\geq (l-2)(k-1)$, which then supports, albeit weaker than, \zcref{supersingular-prime-conjecture}.

We finish this section with an example of \zcref{supersingular-prime-conjecture} when $k=4$ and $l=1$. Since $k=4$, we will simply adopt the normalization $F_{4,C}^{(r)}$ by \zcref{nontilde-normalization}. In this case, one expects that
\begin{align*}
	a_{np}(F_{4,C}^{(1)})\;\equiv\;0\pmod{p^2}\quad\text{ and }\quad
	a_{np}(F_{4,C}^{(2)})\;\equiv\;0\pmod{p}
\end{align*}
for all $n\in\Z^+$, or equivalently,
\begin{align*}
	U_pF_{4,C}^{(1)}\;\equiv\;0\pmod{p^2}\quad\text{ and }\quad
	U_pF_{4,C}^{(2)}\;\equiv\;0\pmod{p}.
\end{align*}
This example is also related to the heuristic given by overconvergent modular forms.

\begin{example}
	Let $k=4$. Let $p=13$ and let $C/\Q$ be the elliptic curve given by $y^2=x^3-x^2-4$ with LMFDB label \href{https://www.lmfdb.org/EllipticCurve/Q/56/b/2}{56.b2}. In particular, $C$ is supersingular at $13$ and $j(C)=-4/7$ satisfies that $v_{13}(j(C))=0=v_{13}(j(C)-1728)$. We would now like to write
	\begin{align*}
		\frac{E_4}{j+4/7}\pmod{13^2}\quad\text{ and }\quad\frac{E_4}{(j+4/7)^2}\pmod{13}
	\end{align*}
	in terms of linear combinations of Hecke eigenforms of level $13$ and weight $4$.
	
	Note that there are exactly three Hecke eigenforms in $S_4(\Gamma_0(13))$: the one $f_1$ with rational coefficients and with LMFDB label \href{https://www.lmfdb.org/ModularForm/GL2/Q/holomorphic/13/4/a/a/}{13.4.a.a}, and the two embeddings $f_2,f_3$ of the one with coefficients in $\Q(\sqrt{17})$ and with LMFDB label \href{https://www.lmfdb.org/ModularForm/GL2/Q/holomorphic/13/4/a/b/}{13.4.a.b}. In particular,
	\begin{align*}
		U_{13}f_1\eq -13 f_1,\quad U_{13}f_2\eq 13f_2,\quad U_{13}f_3\eq 13f_3,
	\end{align*}
	and
	\begin{alignat*}{2}
		f_1&\;\equiv\; q-5q^2-7q^3+17q^4+O(q^5) &&\pmod{13^2} \\
		f_2&\;\equiv\; q+19q^2-53q^3+15q^4+O(q^5) &&\pmod{13^2} \\
		f_3&\;\equiv\; q-18q^2+58q^3-22q^4+O(q^5) &&\pmod{13^2}.
	\end{alignat*}
	Moreover, it is easy to check that
	\begin{align*}
		f_1\;\equiv\;f_3\pmod{13}.
	\end{align*}
	
	For $r=1$, we find that
	\begin{align}
		\label{supersingular-example-r=1}
		\frac{E_4}{j+4/7}\;\equiv\; q+99q^2+32q^3+108q^4+O(q^5)\;\equiv\; 7f_1 -13f_2 +7f_3\pmod{13^2}.
	\end{align}
	Hence,
	\begin{align*}
		U_{13}\bigg(\frac{E_4}{j+4/7}\bigg)\;\equiv\; 13\cdot(-7f_1-13f_2+7f_3)\;\equiv\;0\pmod{13^2}.
	\end{align*}
	For $r=2$, we find that
	\begin{align}
		\label{supersingular-example-r=2}
		\frac{E_4}{(j+4/7)^2}\;\equiv\; q^2+10q^3+q^4+O(q^5)\;\equiv\; 7f_1+6f_2\pmod{13}.
	\end{align}
	Hence, 
	\begin{align*}
		U_{13}\bigg(\frac{E_4}{(j+4/7)^2}\bigg)\;\equiv\; 13\cdot(-7f_1+6f_2)\;\equiv\;0\pmod{13}.
	\end{align*}
	
	In fact, the congruences (\ref{supersingular-example-r=1}) and (\ref{supersingular-example-r=2}) are provable by multiplying both sides by $j+4/7$ or $(j+4/7)^2$ respectively and comparing the Fourier coefficients up to a suitable number of terms. We thus obtain a proof of \zcref{supersingular-prime-conjecture} in this specific case.
\end{example}

\section{The CM case, supercongruences, and magnetic modular forms}
\zlabel{cm-section}

We will now specialize to the CM case, namely when the meromorphic modular form has a pole at a CM point, or equivalently, the elliptic curve $C$ has CM. In this case, we will construct a basis for the space
\begin{align*}
	\MMF_{k,C}\eq\Span\bigg\{\frac{g_k}{j-j(C)},\;\frac{g_k}{(j-j(C))^2},\ldots,\frac{g_k}{(j-j(C))^{k-1}}\bigg\},
\end{align*}
whose elements, after killing the cuspidal part via relations, possess supercongruences and the magnetic property.

We first introduce some notation. Let $D<0$ be a discriminant, $K=\Q(\sqrt{D})$, $\cl{O}_D$ be the imaginary quadratic order of discriminant $D$, and $w_D=|\clo_D^\times|$. Fix a CM point
\begin{align*}
	\alpha_D\eq\frac{D+\sqrt{D}}{2}\;\in\;\clh
\end{align*}
of discriminant $D$. We will assume for simplicity that $D$ is fundamental and has class number~$1$. (See \zcref{non-fund-example} for a brief discussion of the non-fundamental case.) In this case, $D$ is the field discriminant of $K$ and $\clo_D=\clo_K$. With this notation, we also write
\begin{align}
	\MMF_{k,D}\deq\Span\bigg\{\frac{g_k}{j-j(\alpha_D)},\;\frac{g_k}{(j-j(\alpha_D))^2},\ldots,\frac{g_k}{(j-j(\alpha_D))^{k-1}}\bigg\}.
\end{align}
Let $C/\Q$ be an elliptic curve with CM by $\cl{O}_D$ so that $j(C)=j(\alpha_D)$ (and hence $\MMF_{k,C}=\MMF_{k,D}$) and let $\psi_C$ be the corresponding Hecke character of $K$. For convenience, we will choose $C/\Q$ in a way that all of its bad primes divide~$D$. By Deuring's criterion (see \cite{deuring-lifting} or \cite[Chapter~II, Exercise~2.30]{silverman-advanced}), the following are equivalent for a prime $p$ with $p\nmid D$:
\begin{enumerate}[label=(\roman*)]
	\item $p$ is an ordinary prime of $C$ (resp.~$p$ is a supersingular prime of $C$);
	\item $p$ splits in $\cl{O}_D$ (resp.~$p$ is inert in $\cl{O}_D$);
	\item $\big(\frac{D}{p}\big)=1$ (resp.~$\big(\frac{D}{p}\big)=-1$).
\end{enumerate}
We will thus use these three conditions interchangeably.

\subsection{A `canonical' basis and supercongruences}
\zlabel{cm-basis-section}

We will now construct a basis for $\MMF_{k,D}$ which, after killing the cuspidal part via relations, is diagonal (resp.~anti-diagonal) with respect to the action of (powers of) $U_p$ via explicit supercongruences for all ordinary (resp.~supersingular) primes $p$, where the (anti-)eigenvalues will be described explicitly using the $p$-adic analogue of the Chowla--Selberg period.

We first focus on the ordinary case to motivate the construction. In this case, we have the following splitting (as Galois representations of~$K$) on the elliptic curve side
\begin{align}
	\label{cm-elliptic-curve-splitting}
	\Res_{K/\Q}(\Sym^{k-2}H^1_\et(C,\Q_\ell))\;\cong\;\Res_{K/\Q}\big(\Sym^{k-2}(\Ind_{K/\Q}\psi_C)\big)\;\cong\;\bigoplus_{\substack{a+b=k-2 \\ 0\leq a,b\leq k-2}}\psi_C^a\thin\ov{\psi_C^b}.
\end{align}
Let $p$ be a prime that splits in $\cl{O}_D$. Then, for any embedding $\pi:K\hookrightarrow\Q_p$, $\Frob_\pi(=\Frob_p)$ acts on $\psi_C^a\thin\ov{\psi_C^b}$ by $p^bu_p(C)^{a-b}$, where $u_p(C)$ is the root of $X^2-a_p(C)X+p$ which is a $p$-adic unit under the embedding $\pi$.\footnote{Technically, $\Frob_\pi$ acts on one of $\psi_C^a\thin\ov{\psi_C^b}$ and $\psi_C^b\thin\ov{\psi_C^a}$ by $p^bu_p(C)^{a-b}$ depending on the choice of $\psi_C$. We will thus replace $\psi_C$ by $\ov{\psi_C}$ if necessary to ensure that $\Frob_\pi$ acts on $\psi_C^a\thin\ov{\psi_C^b}$ by $p^bu_p(C)^{a-b}$.} On the modular form side, this suggests that the space $\MMF_{k,D}$ should also split into $1$-dimensional pieces $\{V_{a,b}\}_{a+b=k-2}$ equipped with corresponding $U_p$-actions. In other words, there should exist a family $\{G_{a,b}\}_{a+b=k-2}$ of meromorphic modular forms such that
\begin{enumerate}[label=(\arabic*)]
	\item $\Span\{G_{a,b}\mid a+b=k-2\}=\MMF_{k,D}$;
	\item for all (but finitely many) primes $p$ with $\big(\frac{D}{p}\big)=1$, $U_pG_{a,b}\equiv p^bu_p(C)^{a-b} G_{a,b}$ modulo some power of $p$, where $u_p(C)$ is the root of $X^2-a_p(C)X+p$ which is a $p$-adic unit.
\end{enumerate}
In order to construct these $G_{a,b}$, the idea is to view the original $\frac{g_k}{j-j(C)}$ as a two-variable function $\frac{g_k(z)}{j(z)-j(\tau)}$, multiply it by some auxiliary function $f^\aux(\tau)$, differentiate the product (several times) with respect to the auxiliary variable $\tau$, and evaluate $\tau$ at the CM point $\alpha_D$. In fact, the choice of $f^\aux(\tau)$, and hence the two-variable function $\frac{g_k(z)f^\aux(\tau)}{j(z)-j(\tau)}$, was already revealed in \zcref{main-result-section}, which we will recall here.

Let $k\in 2\Z_{\geq 2}$. Recall that
\begin{align}
\label{two-variable-mf}
	G_k(z,\tau)\eq\frac{g_k(z)\cdot g_{2-k}(\tau)}{j(z)-j(\tau)}.
\end{align}
Given a nearly holomorphic modular form $f(\tau)$ of weight $k_f$, recall that its non-holomorphic derivative is defined as
\begin{align}
	(\partial_\tau f)(\tau)\deq \frac{1}{2\pi i}\frac{\partial f}{\partial \tau}-\frac{k_f}{4\pi \Im(\tau)}\cdot f(\tau).
\end{align}
For the fundamental discriminant $D<0$, recall that
\begin{align}
	\label{complex-cs-period}
	\Omega_D\deq\frac{1}{2\pi |D|}\cdot\bigg(\prod_{a=1}^{|D|-1}\Gamma\bigg(\frac{a}{|D|}\bigg)^{(\frac{D}{a})}\bigg)^{w_D/2}.
\end{align}
Indeed, one should be cautious that this is the \textbf{square} of the Chowla--Selberg period that is usually defined in the literature (see e.g.~\cite[p.~84]{zagier-123}). Now, for $1\leq r\leq k-1$, define
\begin{align}
	G_{k,D}^{(r)}(z)\deq(\sqrt{|D|}\cdot\Omega_D)^{\frac{k}{2}-r}\cdot (\partial_\tau^{r-1}G_k)(z,\alpha_D).
\end{align}
While it is not trivial at all, this choice of $G_{k,D}^{(r)}$ appears to automatically take care of the primes $p$ with $v_p(j(\alpha_D))\neq0$ or $v_p(j(\alpha_D)-1728)\neq0$ as well as the cases when $D=-3,-4$. We will thus not exclude them when stating the relevant conjectures.

\begin{lemma}
The modular form $G_{k,D}^{(r)}$ has rational coefficients and is of the form
\begin{align*}
	G_{k,D}^{(r)}\eq g_k\cdot\bigg(\frac{A_1}{j-j(\alpha_D)}+\cdots+\frac{A_r}{(j-j(\alpha_D))^r}\bigg)
\end{align*}
for some $A_i\in\Q$. 
\end{lemma}
\begin{proof}
We first show that $G_{k,D}^{(r)}$ has rational coefficients. Write
\begin{align*}
	G_k(z,\tau)\eq\sum_{n=1}^\infty f_n(\tau)q_z^n,
\end{align*}
where each $f_n(\tau)$ is a weakly holomorphic modular form weight $2-k$ with integer coefficients (see \zcref{G-k-expansion-thm}). As $\partial^{r-1} f_n$ is a nearly weakly holomorphic modular form of weight $2r-k$, we have $\partial^{r-1} f_n=P(E_2^*,E_4,E_6,\Delta^{-1})$ for some $P\in\C[X,Y,Z,W]$. Since each $f_n(\tau)$ has integer coefficients, we must have $P\in\Q[X,Y,Z,W]$. Now, let $E_2^{*a}E_4^bE_6^c\Delta^{-d}$ be a monomial of $P$. Then,
\begin{align*}
	a+2b+3c-6d\eq r-\frac{k}{2}.
\end{align*}
Hence, we have
\begin{align*}
	&(\sqrt{|D|}\cdot\Omega_D)^{\frac{k}{2}-r}\cdot E_2^*(\alpha_D)^aE_4(\alpha_D)^bE_6(\alpha_D)^c \Delta(\alpha_D)^{-d}\eq \\
	&\hspace{3cm}\bigg(\frac{E_2^*(\alpha_D)}{\sqrt{|D|}\cdot\Omega_D}\bigg)^a\cdot \bigg(\frac{E_4(\alpha_D)}{(\sqrt{|D|}\cdot\Omega_D)^2}\bigg)^b\cdot \bigg(\frac{E_6(\alpha_D)}{(\sqrt{|D|}\cdot\Omega_D)^3}\bigg)^c\cdot\bigg(\frac{\Delta(\alpha_D)}{(\sqrt{|D|}\cdot\Omega_D)^6}\bigg)^{-d}.
\end{align*}
Here the rationality of each term in the product follows from CM theory (see e.g.~\cite{cox,masser,milla-s2}), so the whole product is rational.

Now, from the definition of $\partial_\tau$, it is easy to check that
\begin{align*}
	(\partial_\tau^{r-1} G_k)(z,\tau)\eq g_k(z)\cdot\bigg(\frac{t_1(\tau)}{j(z)-j(\tau)}+\cdots\frac{t_r(\tau)}{(j(z)-j(\tau))^r}\bigg)
\end{align*}
for some nearly meromorphic modular form $t_i$ of weight $2r-k$. Hence,
\begin{align*}
	G_{k,D}^{(r)}\eq g_k\cdot\bigg(\frac{A_1}{j-j(\alpha_D)}+\cdots+\frac{A_r}{(j-j(\alpha_D))^r}\bigg)
\end{align*}
with
\begin{align*}
	A_i\eq(\sqrt{|D|}\cdot\Omega_D)^{\frac{k}{2}-r}\cdot t_i(\alpha_D).
\end{align*}
As $G_{k,D}^{(r)}$, $g_k$, and $j$ all have rational coefficients, and $j(\alpha_D)\in\Z$, it then follows from looking at the first few Fourier coefficients of $G_{k,D}^{(r)}$ that each $A_i$ must be a rational number.
\end{proof}

Through the following conjecture, one observes that each $G_{k,D}^{(r)}|\lambda$ takes the role of $G_{k-1-r,r-1}$. In particular, \zcref{ordinary-prime-conjecture} in the CM case is essentially the special case of the following conjecture when $r=1$.

\begin{conjecture}
\zlabel{ordinary-prime-conjecture-cm}
Let $k\in 2\Z_{\geq2}$ and let $1\leq r\leq k-1$. Let $D<0$ be a fundamental discriminant with class number~$1$, $C/\Q$ be an elliptic curve with $j(C)=j(\alpha_D)$, and $\lambda$ be a relation for $S_k$. Let $p$ be a prime with $\big(\frac{D}{p}\big)=1$. Then, for all $n,l\in\Z^+$,
\begin{align}
	a_{np^l}\big(G_{k,D}^{(r)}|\lambda\big)\;\equiv\; p^{r-1}u_p(C)^{k-2r}a_{np^{l-1}}\big(G_{k,D}^{(r)}|\lambda\big)\pmod{p^{(k-1)l}},
\end{align}
where $u_p(C)$ is the root of $X^2-a_p(C)X+p$ which is a $p$-adic unit.
\end{conjecture}
\begin{remark}
It is possible to rewrite this congruence purely in terms of negative discriminants. Since $\big(\frac{D}{p}\big)=1$, we have that $p=\pi\ov{\pi}$ for some $\pi\in\cl{O}_D$. Then, for all $n,l\in\Z^+$,
\begin{align}
	a_{np^l}\big(G_{k,D}^{(r)}|\lambda\big)\;\equiv\; \pi^{k-1-r}\thin\ov{\pi}^{r-1} a_{np^{l-1}}\big(G_{k,D}^{(r)}|\lambda\big)\pmod{\ov{\pi}^{(k-1)l}}.
\end{align}
\end{remark}
\begin{remark}
	As with the case of \zcref{supersingular-prime-conjecture}, it is also possible to deduce the ASD congruence (\zcref{asd-conjecture}) in the ordinary CM case from \zcref{ordinary-prime-conjecture-cm} through involved computations which we omit here.
\end{remark}
\begin{remark}
	One can also provide some explanation for \zcref{supersingular-prime-conjecture} in this CM case through the heuristic using Galois representations. If $p=\fr{p}$ is inert in $\cl{O}_D$, then $\Frob_\fr{p}(=\Frob_p^2)$ acts on $\psi_C^a\thin\ov{\psi_C^b}$ by $p^{k-2}$. Hence, for the potential basis $\{G_{a,b}\}_{a+b=k-2}$, one would expect that $U_p^2G_{a,b}\equiv p^{k-2}G_{a,b}$ modulo some power of $p$ for supersingular primes $p$, as observed in \zcref{supersingular-prime-conjecture}.
\end{remark}

Now we turn to the supersingular primes. In this case, with the help of the $p$-adic Gamma function, one can observe that for a supersingular prime $p$, the family $\{G_{k,D}^{(r)}|\lambda\}_{1\leq r\leq k-1}$ behaves like an anti-diagonal basis with respect to the $U_p$ operator, namely $U_p(G_{k,D}^{(r)}|\lambda)$ is a multiple of $G_{k,D}^{(k-r)}|\lambda$ modulo some power of $p$. Indeed, such phenomenon was also hinted at in \cite[Corollary~1.2 and Example~3]{allen-long-saad}.

For the fundamental discriminant $D<0$ and a prime $p$ with $p\nmid D$, define
\begin{align}
\label{p-adic-cs-period}
	\Omega_{D,p}\deq -\big(\tfrac{D}{p}\big)\cdot\bigg(\prod_{a=1}^{|D|-1}\Gamma_p\bigg(\frac{a}{|D|}\bigg)^{\big(\frac{D}{a}\big)}\thin\bigg)^{w_D/2}\;\in\;\Z_p^\times,
\end{align}
where $\Gamma_p$ denotes the $p$-adic Gamma function and $(\frac{D}{\cdot})$ denotes the Kronecker symbol. One should view this as the $p$-adic analogue of (the square of) the Chowla--Selberg period $\Omega_D$ defined by (\ref{complex-cs-period}).

As in \zcref{higher-pole-section}, a better normalization in the supersingular case is given by
\begin{align}
	\wt{G}_{k,D}^{(r)}\deq(k-1-r)!\cdot G_{k,D}^{(r)}.
\end{align}
We are now ready to describe the conjecture in the supersingular case.

\begin{conjecture}
\zlabel{p-adic-gamma-supersingular-conjecture}
Let $k\in 2\Z_{\geq2}$ and let $1\leq r\leq k-1$. Let $D<0$ be a fundamental discriminant with class number $1$ and $\lambda$ be a relation for $S_k$. Let $p$ be a prime with $\big(\frac{D}{p}\big)=-1$. Then, for all $n,l\in\Z^+$,
	\begin{align}
		a_{np^l}\big(\wt{G}_{k,D}^{(k-r)}|\lambda\big)\;\equiv\;(-p)^{r-1} \Omega_{D,p}^{-(\frac{k}{2}-r)}a_{np^{l-1}}\big(\wt{G}_{k,D}^{(r)}|\lambda\big)\pmod{p^{(k-1)l}}.
	\end{align}
\end{conjecture}

It is also possible to describe the $p$-adic behaviors of $\{G_{k,D}^{(r)}|\lambda\}_{1\leq r\leq k-1}$ for ordinary primes $p$ using the $p$-adic Chowla--Selberg period.

\begin{conjecture}
\zlabel{p-adic-gamma-ordinary-conjecture}
Let $k\in 2\Z_{\geq2}$ and let $1\leq r\leq k-1$. Let $D<0$ be a fundamental discriminant with class number $1$ and $\lambda$ be a relation for $S_k$. Let $p$ be a prime with $\big(\frac{D}{p}\big)=1$. Then, for all $n,l\in\Z^+$,
\begin{align}
	a_{np^l}\big(G_{k,D}^{(r)}|\lambda\big)\;\equiv\; p^{r-1}\Omega_{D,p}^{\frac{k}{2}-r}a_{np^{l-1}}\big(G_{k,D}^{(r)}|\lambda\big)\pmod{p^{(k-1)l}}.
\end{align}
\end{conjecture}

\begin{remark}
\zlabel{renumbering-remark}
The main difference between the supersingular and the ordinary cases is that the index on the left is $k-r$ for the supersingular one and $r$ for the ordinary one. That is, with respect to the action of $U_p$, the family $\{G_{k,D}^{(r)}|\lambda\}_{1\leq r\leq k-1}$ appears diagonal when $p$ is ordinary and anti-diagonal when $p$ is supersingular. In fact, it is possible to write the supersingular and the ordinary cases in a (roughly) uniform way. For $1-\frac{k}{2}\leq i\leq\frac{k}{2}-1$, let
\begin{align*}
	H_{k,D}^{(i)}\deq\wt{G}_{k,D}^{(\frac{k}{2}+i)}.
\end{align*}
Then, for all $n,l\in\Z^+$,
\begin{align}
	a_{np^l}\big(H_{k,D}^{(\epsilon_p i)}|\lambda\big)\;\equiv\;\big(\epsilon_p p)^{\frac{k}{2}+i-1}\Omega_{D,p}^{-\epsilon_p i}\thin a_{np^{l-1}}\big(H_{k,D}^{(i)}|\lambda\big)\pmod{p^{(k-1)l}},
\end{align}
where $\epsilon_p=\big(\frac{D}{p}\big)$.
\end{remark}

\begin{remark}
By a result of Gross--Koblitz \cite[Theorem~4.13]{gross-koblitz}, one can give an explicit characterization of $\Omega_{D,p}$ when $\big(\frac{D}{p}\big)=1$ and $-D$ is not a power of $2$. Let $D<0$ be a fundamental discriminant with class number $1$. Suppose that $p$ splits as $p=\pi\ov{\pi}$ in $\cl{O}_D$ and let $\iota_\pi:\cl{O}_D\rightarrow\Z_p$ be the $p$-adic embedding induced by $\pi$. Then,
	\begin{align*}
		\Omega_{D,p}\eq\iota_\pi(\ov{\pi}^2).
	\end{align*}
	Equivalently, by our choice of $C/\Q$,
	\begin{align*}
		\Omega_{D,p}\eq u_p(C)^2,
	\end{align*}
	where $u_p(C)$ is the root of $X^2-a_p(C)X+p$ which is a $p$-adic unit. In this way, \zcref{p-adic-gamma-ordinary-conjecture} is equivalent to \zcref{ordinary-prime-conjecture-cm} when $-D$ is not a power of $2$.
\end{remark}

\begin{remark}
	\zlabel{l-value-remark}
	One can also relate $\Omega_{D,p}$ with the special value of a $p$-adic $L$-function by a result of Ferrero--Greenberg \cite[Theorem~1]{ferrero-greenberg}. Let $D<0$ be a fundamental discriminant and fix a prime $p$ with $p\nmid D$. Let $\chi=(\frac{D}{\cdot})$ be the Dirichlet character associated to $D$ and let $\eps_p$ be the $p$-adic cyclotomic character. Then, 
	\begin{align}
		\label{p-adic-lerch}
		L_p'(0,\chi\eps_p)\eq \sum_{a=1}^{|D|-1}\chi(a)\log_p\Gamma_p\bigg(\frac{a}{|D|}\bigg)-L_p(0,\chi\eps_p)\cdot \log_p |D|,
	\end{align}
	and hence
	\begin{align*}
		L_p'(0,\chi\eps_p)\eq \frac{2}{w_D}\log_p\Omega_{D,p}-L_p(0,\chi\eps_p)\cdot \log_p |D|,
	\end{align*}
	where $L_p$ denotes the $p$-adic $L$-function and $\log_p$ denotes the $p$-adic logarithm. Here $L_p(0,\chi\eps_p)$ can be explicitly expressed as the following
	\begin{align*}
		L_p(0,\chi\eps_p)\eq -\big(1-\chi(p)\big)\cdot\frac{1}{|D|}\sum_{a=1}^{|D|-1}\chi(a)a.
	\end{align*}
	The formula (\ref{p-adic-lerch}) should be compared with the version of the complex $L$-functions. In this case, Lerch's formula \cite[p.~303,~(25)]{lerch} says that
	\begin{align*}
		L'(0,\chi)&\eq \sum_{a=1}^{|D|-1}\chi(a)\log\Gamma\bigg(\frac{a}{|D|}\bigg)-L(0,\chi)\cdot \log |D|.
	\end{align*}
\end{remark}

\subsection{The magnetic property and other observations}
\zlabel{magnetic-other-observation-section}

In the previous subsection, we have explored properties of the family $\{G_{k,D}^{(r)}\}_{1\leq r\leq k-1}$ by viewing $\Sym^{k-2}C$ as a Galois representation of~$K$. In fact, one can already observe interesting numerical phenomena for the family $\{G_{k,D}^{(r)}\}_{1\leq r\leq k-1}$ by viewing $\Sym^{k-2}C$ simply as a Galois representation of $\Q$. 

Indeed, we have the following splitting (as Galois representations of $\Q$)\footnote{A possibly easier way of observing this splitting is through the splitting of the $L$-function of $\Sym^{k-2}C$.}
\begin{align}
	\Sym^{k-2}H^1_\et(C,\Q_\ell)\;\cong\;\big((\tfrac{D}{\cdot})\thin\eps_\ell\big)^{\frac{k-2}{2}}\oplus\thin\bigoplus_{r=1}^{\frac{k-2}{2}}\big(\eps_\ell^{r-1}\otimes\Ind_{K/\Q}\psi_C^{k-2r}\big),
\end{align}
where $(\frac{D}{\cdot})$ denotes (the Dirichlet character given by) the Kronecker symbol and $\eps_\ell$ denotes the $\ell$-adic cyclotomic character. Here the term $\big((\tfrac{D}{\cdot})\thin\eps_\ell\big)^{\frac{k-2}{2}}$ should correspond to $G_{k,D}^{(k/2)}|\lambda$ and the term $\eps_\ell^{r-1}\otimes\Ind_{K/\Q}\psi_C^{k-2r}$ should correspond to $\Span\{G_{k,D}^{(r)}|\lambda,G_{k,D}^{(k-r)}|\lambda\}$.
\begin{center}
	\begin{tikzcd}[row sep=30pt, nodes={inner sep=2pt}]
		\Sym^{k-2}H_\et^1(C,\Q_\ell)\;\cong\;\big((\tfrac{D}{\cdot})\thin\eps_\ell\big)^{\frac{k-2}{2}}\;\oplus\;\bigoplus\limits_{r=1}^{\frac{k-2}{2}}\big(\eps_\ell^{r-1}\otimes\Ind_{K/\Q}\psi_C^{k-2r}\big) \arrow[d, leftrightarrow, dashed, shift left=90pt] \arrow[d, leftrightarrow, dashed, shift right=25pt]  \\
		\MMF_{k,C}|\lambda\eq\Span\{G_{k,D}^{(k/2)}|\lambda\}\;\oplus\;\bigoplus\limits_{r=1}^{\frac{k-2}{2}}\Span\{G_{k,D}^{(r)}|\lambda,G_{k,D}^{(k-r)}|\lambda\}
	\end{tikzcd}
\end{center}
We first observe that the appearance of powers of $\eps_\ell$ seems to account for the magnetic property of~$G_{k,D}^{(r)}|\lambda$. Specifically, we have the following conjecture.

\begin{conjecture}
\zlabel{cm-magnetic-conjecture}
Let $k\in 2\Z_{\geq2}$ and let $1\leq r\leq k-1$. Let $D<0$ be a fundamental discriminant with class number $1$ and let $\lambda$ be a relation for $S_k$. Let $A=A_{k,D}^{(r)}$ be a nonzero rational number such that $\wh{G}_{k,D}^{(r)}=AG_{k,D}^{(r)}$ has integer coefficients with no common divisors. Then, $\wh{G}_{k,D}^{(r)}|\lambda$ is $(r'-1)$-magnetic, i.e., for all $n\in\Z^+$,
	\begin{align*}
		n^{r'-1}\thin\Big|\thin a_n\big(\wh{G}_{k,D}^{(r)}|\lambda\big),
	\end{align*}
	where $r'=\min(r,k-r)$.
\end{conjecture}
\begin{remark}
	If one only focuses on $a_p=a_p\big(G_{k,D}^{(r)}|\lambda\big)$, then generically one obtains the following table from \zcref{supersingular-prime-conjecture} (or \zcref{p-adic-gamma-supersingular-conjecture}) and \zcref{ordinary-prime-conjecture-cm}.
		\begin{table}[H]
			\renewcommand*{\arraystretch}{1.3}
			\centering
			\begin{tabular}{ |c|c c c c c c| }
				\hline
				$r$ & $1$ & $2$ & $3$ & $\cdots$ & $k-2$ & $k-1$ \\
				\hline
				$v_p(a_p)$ for $\big(\frac{D}{p}\big)=-1$ & $k-2$ & $k-3$ & $k-4$ & $\cdots$ & $1$ & $0$ \\ 
				$v_p(a_p)$ for $\big(\frac{D}{p}\big)=1$ & $0$ & $1$ & $2$ & $\cdots$ & $k-3$ & $k-2$ \\
				\hline
			\end{tabular}
		\end{table}
	\noindent To some extent, this provides a visual explanation of the $(r'-1)$-magnetic property of $G_{k,D}^{(r)}|\lambda$.
\end{remark}

The following two results provide the comparison between $G_{k,D}^{(k/2)}|\lambda$ and $\big((\frac{D}{\cdot})\eps_\ell\big)^{\frac{k-2}{2}}$. We will prove this in \zcref{shimura-magnetic-section} using the Borcherds--Shimura lift. We make a remark that this result is correctly written using the normalization $G_{k,D}^{(k/2)}$ rather than $\wt{G}_{k,D}^{(k/2)}$. In other words, there is no need to distinguish between ordinary and supersingular cases.

\begin{theorem}
\zlabel{cm-middle-supercongruence}
Let $k\in 2\Z_{\geq2}$, $D<0$ be a fundamental discriminant with class number $1$, and $\lambda$ be a relation for $S_k$. Let $p$ be a prime with $p\nmid D$. Then, $G_{k,D}^{(k/2)}$ satisfies that for all $n,l\in\Z^+$,
	\begin{align}
		a_{np^l}\big(G_{k,D}^{(k/2)}|\lambda\big)\;\equiv\;\Big(\big(\tfrac{D}{p}\big)\thin p\Big)^{\frac{k-2}{2}}a_{np^{l-1}}\big(G_{k,D}^{(k/2)}|\lambda\big)\pmod{p^{(k-1)l}}.
	\end{align}
\end{theorem}

\begin{theorem}
\zlabel{cm-middle-magnetic}
Let $k\in 2\Z_{\geq2}$, $D<0$ be a fundamental discriminant with class number $1$, and $\lambda$ be a relation for $S_k$. Let
\begin{align*}
	\wh{G}_{k,D}^{(k/2)}\eq\begin{cases}
		\frac{2}{w_D}|D|^{-\frac{k}{4}}G_{k,D}^{(k/2)}&\text{ if }k\equiv0\pmod{4},\\
		\frac{2}{w_D}|D|^{\frac{k-2}{4}}G_{k,D}^{(k/2)}&\text{ if }k\equiv 2\pmod{4}.
	\end{cases}
\end{align*}
Then, $\wh{G}_{k,D}^{(k/2)}|\lambda$ is $(\frac{k}{2}-1)$-magnetic.
\end{theorem}

To describe the phenomenon for the non-central case, i.e., for $r\neq\frac{k}{2}$, we define a modular form given by a theta series. For $j\in 2\Z_{\geq0}$, let
\begin{align}
	\Theta_{j+1,D}\deq\frac{1}{2}\sum_{\alpha\in\clo_D}\alpha^jq^{|\N(\alpha)|}.
\end{align}
This is a CM Hecke eigenform of weight $j+1$, and it is a cusp form if $j>0$. Now, the following congruence provides the comparison between $\eps_\ell^{r-1}\otimes\Ind_{K/\Q}\psi_C^{k-2r}$ and $\Span\{G_{k,D}^{(r)}|\lambda,G_{k,D}^{(k-r)}|\lambda\}$.

\begin{conjecture}
	\zlabel{asd-cm-conjecture-higher}
	Let $k\in2\Z_{\geq2}$, $1\leq r\leq k-1$ with $r\neq\frac{k}{2}$, and $r'=\min(r,k-r)$. Let $D<0$ be a fundamental discriminant with class number $1$ and $\lambda$ be a relation for $S_k$. Let $p$ be a prime with $p\nmid D$. Then, for all $n,l\in\Z^+$,
	\begin{align*}
		a_{np^l}\big(\wt{G}_{k,D}^{(r)}|\lambda\big)-p^{r'-1}a_p(\Theta_{k+1-2r',D})\thin a_{np^{l-1}}\big(\wt{G}_{k,D}^{(r)}|\lambda\big)+&\big(\tfrac{D}{p}\big)p^{k-2}a_{np^{l-2}}\big(\wt{G}_{k,D}^{(r)}|\lambda\big)\;\equiv\;0\\
		&\hspace{2.5cm}\pmod{p^{(k-1)l-r+\delta_{p,r}}}.
	\end{align*}
\end{conjecture}
\begin{remark}
It is easy to check that this conjecture follows from the supersingular (\zcref{supersingular-prime-conjecture}) and ordinary (\zcref{ordinary-prime-conjecture-cm}) dichotomy.
\end{remark}

The final two observations are related to tensor products. In the splitting (\ref{cm-elliptic-curve-splitting}) of the elliptic curve side over $K$, one has that $\psi_C^a\ov{\psi_C^b}\otimes\psi_C^b\ov{\psi_C^a}=(\psi_C\ov{\psi_C})^{k-2}=(\Res_{K/\Q}\eps_\ell)^{k-2}$. On the modular form side, this suggests that the `tensor product' of $G_{k,D}^{(r)}|\lambda$ and $G_{k,D}^{(k-r)}|\lambda$ should resemble the character $\eps_\ell^{k-2}$. Here we form the `tensor product' of two modular forms via termwise multiplication, motivated by the Rankin--Selberg theory.

\begin{conjecture}
	\zlabel{r-k-r-congruence}
	Let $k\in 2\Z_{\geq2}$, $r\in\Z$ with $1\leq r\leq k-1$, and $r'=\min(r,k-r)$. Let $D<0$ be a fundamental discriminant with class number $1$ and $\lambda$ be a relation for $S_k$. Let $p$ be a prime with $p\nmid D$. Then, for all $n,l\in\Z^+$,
	\begin{equation}
	\begin{aligned}
		\label{r-k-r-congruence-equation}
		a_{np^l}\big(\wt{G}_{k,D}^{(r)}|\lambda\big)\thin a_{np^l}\big(\wt{G}_{k,D}^{(k-r)}|\lambda\big)\;\equiv&\; p^{k-2}a_{np^{l-1}}\big(\wt{G}_{k,D}^{(r)}|\lambda\big)\thin a_{np^{l-1}}\big(\wt{G}_{k,D}^{(k-r)}|\lambda\big) \\
		&\hspace{3cm}\pmod{p^{(k+r'-2)l+\delta_{p,r}+\delta_{p,k-r}}}.
	\end{aligned}
	\end{equation}
\end{conjecture}

\begin{remark}
When $\big(\frac{D}{p}\big)=-1$, the numerical observations suggest that the congruence (\ref{r-k-r-congruence-equation}) holds modulo $p^{(k+r'-2)l+(k-2r')}$ for $l\geq 2$.
\end{remark}

As before, the appearance of $\eps_\ell$ seems to provide some magnetic property for the `tensor product' of $G_{k,D}^{(r)}|\lambda$ and $G_{k,D}^{(k-r)}|\lambda$.

\begin{conjecture}
Let $k\in 2\Z_{\geq2}$, $r\in\Z$ with $1\leq r\leq k-1$, $D<0$ be a fundamental discriminant with class number $1$, and $\lambda$ be a relation for $S_k$. Then, there exists a nonzero integer $A$ such that for all $n\in\Z^+$,
	\begin{align*}
		n^{k-2}\thin\Big|\thin A\thin a_n\big(G_{k,D}^{(r)}|\lambda\big)\thin a_n\big(G_{k,D}^{(k-r)}|\lambda\big).
	\end{align*}
\end{conjecture}

\subsection{Numerical examples}
\zlabel{cm-example-section}

We will now discuss some numerical examples. We will mostly focus on supercongruences involving the $p$-adic Chowla--Selberg period $\Omega_{D,p}$, namely \zcref{p-adic-gamma-supersingular-conjecture} and \zcref{p-adic-gamma-ordinary-conjecture}, since other phenomena are more or less consequences of these two supercongruences. In all of the examples, we typically compute the first $5000$ coefficients of the relevant meromorphic modular form and check the corresponding supercongruences numerically.

\begin{example}
\zlabel{k=4-D=3-4-example}
We first discuss three examples when $D=-3,-4$.
\begin{enumerate}[label=(\arabic*)]
	\item Let $k=4$ and $D=-3$ so that $j(\alpha_D)=0$. In this case, $G_{4,-3}^{(1)}=G_{4,-3}^{(3)}=0$, and
	\begin{align*}
		G_{4,-3}^{(2)}\eq 576\cdot\frac{E_4}{j}.
	\end{align*}
	Now, it follows from \zcref{cm-middle-supercongruence} that for all primes $p\geq5$ and $n,l\in\Z^+$
	\begin{align}
		a_{np^l}\bigg(\frac{E_4}{j}\bigg)\;\equiv\;\big(\tfrac{-3}{p}\big)p\thin a_{np^{l-1}}\bigg(\frac{E_4}{j}\bigg)\pmod{p^{3l}}.
	\end{align}
	Moreover, it follows from \zcref{cm-middle-magnetic} that $64\cdot\frac{E_4}{j}$ is $1$-magnetic. Note that this is exactly \cite[Theorem~1.5]{li-neururer} with the same pre-factor $64$. This pre-factor was removed later in \cite{pasol-zudilin}.
	\item Let $k=4$ and $D=-4$ so that $j(\alpha_D)=1728$. In this case, $G_{4,-4}^{(1)}=G_{4,-4}^{(3)}=0$, and
	\begin{align*}
		G_{4,-4}^{(2)}\eq -864\cdot\frac{E_4}{j-1728}.
	\end{align*}
	Now, it follows from \zcref{cm-middle-supercongruence} that for all primes $p\geq5$ and $n,l\in\Z^+$
	\begin{align}
		a_{np^l}\bigg(\frac{E_4}{j-1728}\bigg)\;\equiv\;\big(\tfrac{-4}{p}\big)p\thin a_{np^{l-1}}\bigg(\frac{E_4}{j-1728}\bigg)\pmod{p^{3l}}.
	\end{align}
	Moreover, it follows from \zcref{cm-middle-magnetic} that $108\cdot\frac{E_4}{j-1728}$ is $1$-magnetic, though the pre-factor $108$ can be removed by \cite{pasol-zudilin}.
	\item Let $k=6$ and $D=-3$ so that $j(\alpha_D)=0$. In this case, $G_{6,-3}^{(r)}=0$ for $r\neq 3$, and
	\begin{align*}
		G_{6,-3}^{(3)}\eq -384\cdot\frac{E_6}{j}.
	\end{align*}
	Now, it follows from \zcref{cm-middle-supercongruence} that for all primes $p\geq5$ and $n,l\in\Z^+$
	\begin{align}
		a_{np^l}\bigg(\frac{E_6}{j}\bigg)\;\equiv\;\big(\tfrac{-3}{p}\big)p\thin 	a_{np^{l-1}}\bigg(\frac{E_6}{j}\bigg)\pmod{p^{5l}}.
	\end{align}
	Moreover, it follows from \zcref{cm-middle-magnetic} that $384\cdot\frac{E_6}{j}$ is $1$-magnetic, though the pre-factor $384$ can be removed by \cite{pasol-zudilin}.
\end{enumerate}
\end{example}

\begin{example}
	Let $k=4$ and let $D=-7$ so that $j(\alpha_D)=-3375$. In this case,
	\begin{align*}
		G_{4,-7}^{(1)}&\eq-2835\cdot\frac{E_4}{j+3375} \\
		G_{4,-7}^{(2)}&\eq189\cdot\bigg(19\cdot\frac{E_4}{j+3375}-91125\cdot\frac{E_4}{(j+3375)^2}\bigg) \\
		G_{4,-7}^{(3)}&\eq-\frac{27}{7}\cdot\bigg(1399\cdot\frac{E_4}{j+3375}-19008675\cdot\frac{E_4}{(j+3375)^2}+54251268750\cdot\frac{E_4}{(j+3375)^3}\bigg).
	\end{align*}
	Here it follows from \zcref{cm-middle-magnetic} that $\frac{1}{7}G_{4,-7}^{(2)}=27\cdot\big(19\cdot\frac{E_4}{j+3375}-91125\cdot\frac{E_4}{(j+3375)^2}\big)$ is $1$-magnetic, while it is already numerically $1$-magnetic without the pre-factor $27$.
	
	To give concrete examples of the supercongruences, we consider $p=2$ and $p=3$, with $(\frac{-7}{2})=1$ and $(\frac{-7}{3})=-1$.
	
	For $p=2$, one observes numerically that for $n\in\Z^+$,
	\begin{alignat*}{2}
		a_{2n}(G_{4,-7}^{(1)})&\;\equiv\; \Omega_{-7,2}\cdot a_n(G_{4,-7}^{(1)})&&\pmod{2^7} \\
		a_{2n}(G_{4,-7}^{(2)})&\;\equiv\; 2\thin a_n(G_{4,-7}^{(2)})&&\pmod{2^{10}} \\
		a_{2n}(G_{4,-7}^{(3)})&\;\equiv\; 2^2\Omega_{-7,2}^{-1}\cdot a_n(G_{4,-7}^{(3)})&&\pmod{2^{10}},
	\end{alignat*}
	all of which appear stronger than the predicted ones modulo $2^3$. To test the predicted congruences modulo $2^3$ more explicitly, one may plug in
	\begin{align*}
		\Omega_{-7,2}\eq 1+O(2^3)\;\equiv\; 1\pmod{2^3}.
	\end{align*}
	
	For $p=3$, since the pre-factor of $G_{4,-7}^{(r)}$ is already divisible by $3^3$, we will thus turn to $3^6$. In this case, one observes numerically that for $n\in\Z^+$,
	\begin{alignat*}{3}
		a_{9n}(G_{4,-7}^{(3)})&\;\equiv\; \Omega_{-7,3}^{-1}\cdot a_{3n}(2!\cdot G_{4,-7}^{(1)})&&\pmod{3^7} \\
		a_{9n}(G_{4,-7}^{(2)})&\;\equiv\; -3\thin a_{3n}(G_{4,-7}^{(2)}) &&\pmod{3^7} \\
		a_{9n}(2!\cdot G_{4,-7}^{(1)})&\;\equiv\; (-3)^2\Omega_{-7,3}\cdot a_{3n}(G_{4,-7}^{(3)})&&\pmod{3^8},
	\end{alignat*}
	all of which appear stronger than the predicted ones modulo $3^6$. To test the predicted congruences modulo $3^6$ more explicitly, one may plug in
	\begin{align*}
		\Omega_{-7,3}\eq 1 + 2\cdot 3^2 + 3^3 + 3^4 + 2\cdot 3^5 + O(3^6)\;\equiv\; 613\pmod{3^6}.
	\end{align*}
\end{example}

\begin{example}
Let $k=6$ and let $D=-8$ so that $j(\alpha_D)=8000$. In this case,
{\small
\begin{align*}
	G_{6,-8}^{(1)}&\eq3200\cdot\frac{E_6}{j-8000} \\
	G_{6,-8}^{(2)}&\eq-3520\cdot\frac{E_6}{j-8000}-35840000\cdot\frac{E_6}{(j-8000)^2} \\
	G_{6,-8}^{(3)}&\eq5096\cdot\frac{E_6}{j-8000}+141312000\cdot\frac{E_6}{(j-8000)^2}+802816000000\cdot\frac{E_6}{(j-8000)^3} \\
	G_{6,-8}^{(4)}&\eq-8014\cdot\frac{E_6}{j-8000}-495200000\cdot\frac{E_6}{(j-8000)^2}-6846873600000\cdot\frac{E_6}{(j-8000)^3} \\
	&\hspace{1cm}-26974617600000000\cdot\frac{E_6}{(j-8000)^4} \\
	G_{6,-8}^{(5)}&\eq\frac{56723}{4}\cdot\frac{E_6}{j-8000}+1717555200\cdot\frac{E_6}{(j-8000)^2}+44014796800000\cdot\frac{E_6}{(j-8000)^3}\\
	&\hspace{1cm}+400765747200000000\cdot\frac{E_6}{(j-8000)^4}+1208462868480000000000\cdot\frac{E_6}{(j-8000)^5}.
\end{align*}}%
Here it is easy to check that $G_{6,-8}^{(2)}$ and $G_{6,-8}^{(4)}$ are numerically $1$-magnetic, and it follows from \zcref{cm-middle-magnetic} that $8G_{6,-8}^{(3)}$ is $2$-magnetic. 

To give concrete examples of the supercongruences, we consider $p=11$ and $p=7$, with $(\frac{-8}{11})=1$ and $(\frac{-8}{7})=-1$. For simplicity, we only focus on $G_{6,-8}^{(2)}$ and $G_{6,-8}^{(4)}$.

For $p=11$, one observes numerically that for $n\in\Z^+$,
\begin{alignat*}{2}
	a_{11n}(G_{6,-8}^{(2)})&\;\equiv\; 11\Omega_{-8,11}\cdot a_n(G_{6,-8}^{(2)})&&\pmod{11^5} \\
	a_{11n}(G_{6,-8}^{(4)})&\;\equiv\; 11^3\Omega_{-8,11}^{-1}\cdot a_n(G_{6,-8}^{(4)})&&\pmod{11^5},
\end{alignat*}
as predicted. To test these congruences more explicitly, one may plug in
\begin{align*}
	\Omega_{-8,11}\eq 3 + 11 + 7\cdot 11^2 + 8\cdot 11^3 + 2\cdot 11^4 + O(11^5)\;\equiv\;40791\pmod{11^5}.
\end{align*}

For $p=7$, one observes numerically that for $n\in\Z^+$,
\begin{alignat*}{2}
	a_{7n}(G_{6,-8}^{(4)})&\;\equiv\; (-7)\Omega_{-8,7}^{-1}\cdot a_n(3!\cdot G_{6,-8}^{(2)})&&\pmod{7^6} \\
	a_{7n}(3!\cdot G_{6,-8}^{(2)})&\;\equiv\; (-7)^3\Omega_{-8,11}\cdot a_n(G_{6,-8}^{(4)})&&\pmod{7^6},
\end{alignat*}
both of which appear stronger than the predicted ones modulo $7^5$. To test the predicted congruences modulo $7^5$ more explicitly, one may plug in
\begin{align*}
	\Omega_{-8,7}\eq 4 + 5\cdot 7 + 7^3 + 4\cdot 7^4 + O(7^5)\;\equiv\;9986\pmod{7^5}.
\end{align*}
\end{example}

Let us also give two examples when $k=8$. Since the coefficients are already very large in this case, we will only consider $D=-3,-4$.

\begin{example}
Let $k=8$ and let $D=-3$ so that $j(\alpha_D)=0$. In this case, 
\begin{align*}
	G_{8,-3}^{(1)}&\eq -216\cdot\frac{E_8}{j} \\
	G_{8,-3}^{(4)}&\eq 216\cdot\frac{E_8}{j}-663552\cdot\frac{E_8}{j^2} \\
	G_{8,-3}^{(7)}&\eq -\frac{3632}{3}\cdot\frac{E_8}{j}+21749760\cdot\frac{E_8}{j^2}-40768634880\cdot\frac{E_8}{j^3},
\end{align*}
and $G_{8,-3}^{(r)}=0$ for all the other cases. Here it follows from \zcref{cm-middle-magnetic} that $\frac{1}{27}G_{8,-3}^{(4)}$ is $3$-magnetic.

To give concrete examples of the supercongruences, we consider $p=7$ and $p=11$, with $(\frac{-3}{7})=1$ and $(\frac{-3}{11})=-1$. We will also turn to the congruences modulo $p^{14}$, namely $l=2$.

For $p=7$, one observes numerically that for $n\in\Z^+$,
\begin{alignat*}{2}
	a_{49n}(G_{8,-3}^{(1)})&\;\equiv\; \Omega_{-3,7}^3\cdot a_{7n}(G_{8,-3}^{(1)})&&\pmod{7^{14}} \\
	a_{49n}(G_{8,-3}^{(4)})&\;\equiv\; 7^3\cdot a_{7n}(G_{8,-3}^{(4)})&&\pmod{7^{15}} \\
	a_{49n}(G_{8,-3}^{(7)})&\;\equiv\; 7^6\Omega_{-3,7}^{-3}\cdot a_{7n}(G_{8,-3}^{(7)})&&\pmod{7^{15}},
\end{alignat*}
where the latter two congruences appear stronger than the predicted ones modulo $7^{14}$. To test the predicted congruences modulo $7^{14}$ more explicitly, one may plug in
\begin{align*}
	\Omega_{-3,7}\;\equiv\;138287445423\pmod{7^{14}}.
\end{align*}

For $p=11$, one observes numerically that for $n\in\Z^+$,
\begin{alignat*}{2}
	a_{121n}(G_{8,-3}^{(7)})&\;\equiv\; \Omega_{-3,11}^{-3}\cdot a_{11n}(6!\cdot G_{8,-3}^{(1)})&&\pmod{11^{14}} \\
	a_{121n}(G_{8,-3}^{(4)})&\;\equiv\; (-11)^3\cdot a_{11n}(G_{8,-3}^{(4)})&&\pmod{11^{14}} \\
	a_{121n}(6!\cdot G_{8,-3}^{(1)})&\;\equiv\; (-11)^6\Omega_{-3,11}^3\cdot a_{11n}(G_{8,-3}^{(7)})&&\pmod{11^{14}},
\end{alignat*}
as predicted. To test these congruences more explicitly, one may plug in
\begin{align*}
	\Omega_{-3,11}\;\equiv\;206885105461149\pmod{11^{14}}.
\end{align*}
\end{example}

\begin{example}
	Let $k=8$ and let $D=-4$ so that $j(\alpha_D)=1728$. In this case,
	\begin{align*}
		G_{8,-4}^{(2)}&\eq -288\cdot\frac{E_8}{j-1728} \\
		G_{8,-4}^{(4)}&\eq -384\cdot\frac{E_8}{j-1728}-2239488\cdot\frac{E_8}{(j-1728)^2} \\
		G_{8,-4}^{(6)}&\eq -1077\cdot\frac{E_8}{j-1728}-24261120\cdot\frac{E_8}{(j-1728)^2}-58047528960\cdot\frac{E_8}{(j-1728)^3},
	\end{align*}
	and $G_{8,-4}^{(r)}=0$ if $r$ is odd. Here it is easy to check that $G_{8,-4}^{(2)}$ and $G_{8,-4}^{(6)}$ are numerically $1$-magnetic, and it follows from \zcref{cm-middle-magnetic} that $\frac{1}{32}G_{8,-4}^{(4)}$ is $3$-magnetic.
	
	To give concrete examples of the supercongruences, we consider $p=13$ and $p=3$, with $(\frac{-4}{13})=1$ and $(\frac{-4}{3})=-1$. 
	
	For $p=13$, one observes numerically that for $n\in\Z^+$,
	\begin{alignat*}{2}
		a_{13n}(G_{8,-4}^{(2)})&\;\equiv\; 13\Omega_{-4,13}^2\cdot a_{n}(G_{8,-3}^{(2)})&&\pmod{13^7} \\
		a_{13n}(G_{8,-4}^{(4)})&\;\equiv\; 13^3\cdot a_{n}(G_{8,-3}^{(4)})&&\pmod{13^7} \\
		a_{13n}(G_{8,-4}^{(6)})&\;\equiv\; 13^5\Omega_{-4,13}^{-2}\cdot a_{n}(G_{8,-3}^{(6)})&&\pmod{13^7},
	\end{alignat*}
	as predicted. To test these congruences more explicitly, one may plug in
	\begin{align*}
		\Omega_{-4,13}\eq 10 + 9\cdot 13^2 + 11\cdot 13^3 + 4\cdot 13^4 + 9\cdot 13^6 + O(13^7)\;\equiv\;43581223\pmod{13^7}.
	\end{align*}

	For $p=3$, we turn to the congruences modulo $3^{21}$, namely $l=3$. In this case, one observes numerically that for $n\in\Z^+$,
	\begin{alignat*}{2}
		a_{27n}(G_{8,-4}^{(6)})&\;\equiv\; (-3)\Omega_{-4,3}^{-2}\cdot a_{9n}(5!\cdot G_{8,-3}^{(2)})&&\pmod{3^{23}} \\
		a_{27n}(G_{8,-4}^{(4)})&\;\equiv\; (-3)^3\cdot a_{9n}(G_{8,-3}^{(4)})&&\pmod{3^{23}} \\
		a_{27n}(5!\cdot G_{8,-4}^{(2)})&\;\equiv\; (-3)^5\Omega_{-4,3}^2\cdot a_{9n}(G_{8,-3}^{(6)})&&\pmod{3^{24}},
	\end{alignat*}
	all of which appear stronger than the predicted ones modulo $3^{21}$. To test the predicted congruences modulo $3^{21}$ more explicitly, one may plug in
	\begin{align*}
		\Omega_{-4,3}\;\equiv\;4949244661\pmod{3^{21}}.
	\end{align*}
\end{example}

We finish this section with an example with $D=-28$, a non-fundamental discriminant with class number $1$, which may explain that the requirement that $D$ is fundamental is simply a technical assumption rather than an essential one. 

\begin{example}
\zlabel{non-fund-example}
	Let $k=4$ and let $D=-28$ so that $j(\alpha_D)=255^3$. In this case, one considers the Chowla--Selberg periods, both the complex and the $p$-adic versions, associated to $-7$.
	
	For $1\leq r\leq 3$, let
	\begin{align*}
		G_{4,-28}^{(r)}\deq\bigg(\frac{\sqrt{7}}{2}\cdot\Omega_{-7}\bigg)^{2-r}\cdot(\partial_\tau^{r-1}G_4)(z,\alpha_{-28}).
	\end{align*}
	Explicitly, we have
	\begin{align*}
		G_{4,-28}^{(1)}&\eq 5494230\cdot\frac{E_4}{j-255^3} \\
		G_{4,-28}^{(2)}&\eq -15583239\cdot\frac{E_4}{j-255^3}-274913344274625\cdot\frac{E_4}{(j-255^3)^2} \\
		G_{4,-28}^{(3)}&\eq\frac{659642373}{14}\cdot\frac{E_4}{j-255^3}+\frac{4878007584576975}{2}\cdot\frac{E_4}{(j-255^3)^2}\\
		&\hspace{4.3cm}+27511533685433078071875\cdot\frac{E_4}{(j-255^3)^3}.
	\end{align*}
	Now, one can check numerically that the family $\{G_{4,-28}^{(r)}\}_{1\leq r\leq 3}$ satisfies all the conjectures presented in \zcref{cm-basis-section} and \zcref{magnetic-other-observation-section}, and in particular, the supercongruences in \zcref{p-adic-gamma-supersingular-conjecture} and \zcref{p-adic-gamma-ordinary-conjecture}, where one plugs in the $p$-adic Chowla--Selberg period associated to $-7$.
	
	In general, for the non-fundamental discriminants, one may expect all the phenomena described in \zcref{cm-basis-section} and \zcref{magnetic-other-observation-section} to hold for the family $\{(\partial_\tau^{r-1}G_k)(z,\tau_D)\}_{1\leq r\leq k-1}$ under certain appropriate normalization, and the main difficulty may lie in finding the explicit formulation.
\end{example}

\section{Notation and preliminaries for the proofs}
\label{notation-section}

In the last two sections of this paper, we will demonstrate the proofs of \zcref{good-prime-theorem-intro} and \zcref{cm-middle-supercongruence-intro}. For convenience and for the coherence of the proof, we define some notation and recall some preliminaries in this section.

\subsection{Modular forms of integral weight}

Fix $k\in\Z$. Let $M_k$ denote the space of holomorphic modular forms of weight $k$ and level $1$, $S_k$ denote the space of holomorphic cusp forms of weight~$k$ and level $1$, and $M_k^!$ denote the space of weakly holomorphic modular forms of weight $k$ and level $1$.

Now, write $k=k'+12\delta$ for some $k'\in\{0,4,6,8,10,14\}$ and $\delta\in\Z$. For each $n\in\Z$ with $n\geq -\delta$, there exists a unique weakly holomorphic modular form $g_{k,n}$ of weight $k$ and level $1$ satisfying that
\begin{align*}
    g_{k,n}\eq q^{-n}+O(q^{\delta+1})\in\Z[[q,q^{-1}]].
\end{align*}
It is easy to check that $g_{k,-\delta}=g_k$ and that $\{g_{k,n}\}_{n\geq-\delta}$ forms a basis of $M_k^!$.

\subsection{Modular forms of half-integral weight}
Fix $s\in\Z$. Let $M_{s+1/2}^{!,+}$ denote the Kohnen plus space, i.e., the space of weakly holomorphic modular forms of weight $s+1/2$ and level $\Gamma_0(4)$ whose $n$-th Fourier coefficient is zero if $(-1)^sn\not\equiv0,1\pmod{4}$. Let $S_{s+1/2}^+$ denote the Kohnen plus cuspidal subspace.

Now, write $2s=k'+12\delta$ for some $k'\in\{0,4,6,8,10,14\}$ and $\delta\in\Z$. Let
\begin{align*}
    A\deq\begin{cases}
        2\delta-(-1)^s &\text{ if $\delta$ is odd,} \\
        2\delta &\text{ if $\delta$ is even.}
    \end{cases}
\end{align*}
Then, for each $m\in\Z$ with $m\geq -A$ and $(-1)^{s-1}m\equiv0,1\pmod{4}$, there exists a unique $f_{s+1/2,m}\in M_{s+1/2}^{!,+}$ such that
\begin{align*}
    f_{s+1/2,m}\eq q^{-m}+O(q^{A+1})\in\Z[[q,q^{-1}]].
\end{align*}
In particular, $\{f_{s+1/2,m}\}_m$ forms a basis of $M_{s+1/2}^{!,+}$ (see \cite[Section~2~and~Appendix]{duke-jenkins-basis} for a discussion of this fact).

\subsection{Hecke operators (on $q$-series)}

Let $f=\sum_{n\gg-\infty}a_n(f)q^n$ be a $q$-series. Let $k\in\Z$ and $p$ be a rational prime. Define the Hecke operators (on $q$-series) as
\begin{align*}
    f|U_p\deq\sum_{n\gg-\infty} a_{np}(f)q^n,\quad f|V_p\deq\sum_{n\gg-\infty} a_n(f)q^{np},\quad f|T_{p,k}\deq f|(U_p+p^{k-1}V_p).
\end{align*}
More generally, for $m\in\Z^+$, define 
\begin{align*}
    f|T_{m,k}\deq\sum_{n\gg-\infty}\bigg(\sum_{d>0,d|(m,n)}d^{k-1}a_{mn/d^2}(f)\bigg)q^n
\end{align*}
It is easy to check that $T_{m,k}$ and $T_{n,k}$ commute for any $m,n\in\Z^+$ and $k\in\Z$.

For a map $\chi:\Z\rightarrow\C$ (usually a Dirichlet character), define
\begin{align*}
    f|\chi\deq\sum_{n\gg-\infty}\chi(n)a_n(f)q^n.
\end{align*}
Now, let $s\in\Z$ and let
\begin{align*}
    f\eq\sum_{(-1)^sn\equiv0,1\;(\mathrm{mod}\;4)}a_n(f)q^n\in M_{s+1/2}^{+,!}.
\end{align*}
For a rational prime $p$, define
\begin{align*}
    f|T_{p,s+1/2}\deq\sum_{(-1)^sn\equiv0,1\;(\mathrm{mod}\;4)}\big(a_{np^2}(f)+p^{s-1}\big(\tfrac{(-1)^sn}{p}\big)a_n(f)+p^{2s-1}a_{n/p^2}(f)\big)q^n.
\end{align*}
Indeed, if $p$ is odd, then
\begin{align*}
    f|T_{p,s+1/2}\eq f|(U_p^2+p^{s-1}\chi_p'+p^{2s-1}V_p^2),
\end{align*}
where $\chi_p'(n)=\big(\frac{(-1)^sn}{p}\big)$.

\subsection{Negative discriminants}

Let $D<0$ be a discriminant, $\cl{O}_D$ be the imaginary quadratic order of discriminant $D$, and
\begin{align*}
    w_D\deq|\cl{O}_D^\times|\eq\begin{cases}
        4&\text{ if $D=-4$,} \\
        6&\text{ if $D=-3$,} \\
        2&\text{ otherwise.}
    \end{cases}
\end{align*}
Let $\cl{Q}_D$ be a complete set of $\SL_2(\Z)$-inequivalent positive definite quadratic forms $Q(X,Y)\in\Z[X,Y]$ of discriminant $D$ and let $\cl{Q}_D^\prim\subseteq\cl{Q}_D$ be the subset consisting of primitive elements, i.e., $aX^2+bXY+cY^2$ with $(a,b,c)=1$. For $Q(X,Y)=aX^2+bXY+cY^2\in\cl{Q}_D$, define
\begin{align*}
    \alpha_Q\deq\frac{-b+\sqrt{D}}{2a}\in\clh
\end{align*}
and
\begin{align*}
    w_Q\deq\begin{cases}
        2&\text{ if $Q$ is $\SL_2(\Z)$-equivalent to $a(X^2+Y^2)$,} \\
        3&\text{ if $Q$ is $\SL_2(\Z)$-equivalent to $a(X^2+XY+Y^2)$,} \\
        1&\text{ otherwise.}
    \end{cases}
\end{align*}

\section{Proof of \zcref{hypergeom-theorem-intro} and \zcref{good-prime-theorem-intro}}
\label{good-prime-proof-section}

We will now prove \zcref{good-prime-theorem-intro}. As we will work with general number fields, we first introduce some notation.

Let $L$ be a number field and $\cl{O}_L$ be its ring of integers. For a prime $\fr{p}$ of $L$, let $\F_\fr{p}$ denote its residue field and let $\N(\fr{p}):=|\F_\fr{p}|$ denote its norm. For an elliptic curve $C/L$ and a good prime $\fr{p}$ of $C$, define
\begin{align*}
    a_\fr{p}\deq\N(\fr{p})+1-|C(\F_\fr{p})|.
\end{align*}

Now, for a holomorphic modular form $g\in\cl{O}_L[[q]]$ of weight $k$ and level $1$ and an elliptic curve $C/L$, write
\begin{align*}
    F_{g,C}\deq\frac{g}{j-j(C)},
\end{align*}
and for $k\in 2\Z_{\geq2}$, write
\begin{align*}
    F_{k,C}\deq F_{g_k,C}\eq\frac{g_k}{j-j(C)}.
\end{align*}
The main theorem that we will prove in this section is the following.

\begin{theorem}
\label{good-prime-theorem-number-field}
Let $k\in 2\Z_{\geq2}$, $L$ be a number field, $g\in\cl{O}_L[[q]]$ be a holomorphic modular form of weight $k$ and level $1$, and $C/L$ be an elliptic curve. Let $\lambda$ be a relation for $S_k$ and let $\fr{p}$ be a good prime of $C$ with $\fr{p}\nmid 6$ and $v_\fr{p}(j(C))=0=v_\fr{p}(j(C)-1728)$. Then, for all $n\in\Z^+$,
\begin{align*}
    a_{n\cdot\N(\fr{p})}(F_{g,C}|\lambda)\;\equiv\;a_\fr{p}(C)^{k-2}a_n(F_{g,C}|\lambda)\pmod{\fr{p}}.
\end{align*}
\end{theorem}
\begin{remark}
By \zcref{one-choice-for-k}, it suffices to prove this for $g=g_k$, i.e., for $F_{k,C}$. Indeed, there exist some $\beta\in\cl{O}_L$ such that $F_{g,C}-\beta F_{k,C}\in S_k$. Hence, $F_{g,C}|\lambda=\beta\cdot F_{k,C}|\lambda$.
\end{remark}

As mentioned in the introduction, the proof depends on the hypergeometric interpretation of the coefficients in the case when $k=4$. In particular, we will show the following generalization of \zcref{hypergeom-theorem-intro}.

\begin{theorem}
\label{k=4-hypergeom-theorem}
Let $p$ be a rational prime with $p\nmid 6$ and let $c\in\ov{\Z_p}$ with $v_p(c)=0$. Then, for all $l\in\Z^+$,
\begin{align*}
    a_{p^l}\bigg(\frac{E_4}{j-c}\bigg)&\;\equiv\;\big(c(c-1728)\big)^{\frac{p^l-1}{2}}\cdot\sum_{m=0}^{p^l-1}\frac{(6m)!}{(m!)^3(3m)!}\thin c^{-m} \pmod{p}.
\end{align*}
\end{theorem}
 
\subsection{The two variable function}

From now on, we fix $k\in 2\Z_{\geq2}$, and write $k=k'+12\delta$ for some $k'\in\{0,4,6,8,10,14\}$ and $\delta\in\Z$. Recall that
\begin{align*}
    G_k(z,\tau)\eq\frac{g_k(z)\cdot g_{2-k}(\tau)}{j(z)-j(\tau)}.
\end{align*}
Let $q_z=e^{2\pi i z}$ and $q_\tau=e^{2\pi i\tau}$. Then, we have the following result by Duke--Jenkins, from which they deduced Zagier's duality. 

\begin{theorem}[{\cite[Theorem~2]{duke-jenkins-zagierduality}}]
\label{G-k-expansion-thm}
For any even integer $k\in\Z^+$,
\begin{align}
\label{G-k-expansion}
    \sum_{n=\delta+1}^\infty g_{2-k,n}(\tau)q_z^n\eq G_k(z,\tau)\eq-\sum_{n=-\delta}^\infty g_{k,n}(z)q_\tau^n.
\end{align}
\end{theorem}

Fix a rational prime $p$ and let $l\in\Z^+$. It is easy to check from the definitions that
\begin{align*}
    T_{p^l,k}\eq\sum_{a=0}^l p^{a(k-1)}U_p^{l-a}V_p^a.
\end{align*}
Now, fix $n\in\Z$ with $n\geq -\delta$ and consider $g_{k,n}(z)|T_{p^l,k}$. If $n\geq 0$, then we have
\begin{align*}
    g_{k,n}(z)\big|T_{p^l,k}\eq \sum_{a=0}^l p^{a(k-1)}q_z^{-np^{2a-l}} +O(q^{\delta+1}),
\end{align*}
and it follows that
\begin{align}
\label{g-k-n-under-hecke-greater-0}
    g_{k,n}(z)\big|T_{p^l,k}-\sum_{a=0}^l p^{a(k-1)}g_{k,np^{2a-l}}(z)\in S_k,
\end{align}
where we abuse the notation by setting that $q_z^\alpha:=0$ if $\alpha\notin\Z$, and that $g_{k,\alpha}:=0$ if $\alpha\notin\Z$, for convenience. If $-\delta\leq n<0$, then we have
\begin{align}
\label{g-k-n-under-hecke-less-0}
    g_{k,n}(z)\in S_k\quad\text{ and }\quad g_{k,n}(z)\big|T_{p^l,k}\in S_k.
\end{align}

Now, let $\lambda$ be a relation for $S_k$. In the following computations, the relation $\lambda$ and all the Hecke operators are viewed as acting on the variable $z$. We have
\begin{align*}
    G_k(z,\tau)\big|\lambda T_{p^l,k}\eq G_k(z,\tau)\big|T_{p^l,k}\thin\lambda
    &\req{(\ref{G-k-expansion})} -\sum_{n=-\delta}^\infty \big(g_{k,n}(z)\big|T_{p^l,k}\thin\lambda\big) q_\tau^n \\
    &\req{(\ref{g-k-n-under-hecke-less-0})} -\sum_{n=0}^\infty \big(g_{k,n}(z)\big|T_{p^l,k}\thin\lambda\big) q_\tau^n \\
    &\req{(\ref{g-k-n-under-hecke-greater-0})} -\sum_{n=0}^\infty\bigg(\sum_{a=0}^l p^{a(k-1)}g_{k,np^{2a-l}}(z)\big|\lambda\bigg) q_\tau^n \\
    &\;\equiv\;-\sum_{n=0}^\infty \big(g_{k,np^{-s}}(z)\big|\lambda\big) q_\tau^n \\
    &\requiv{(\ref{g-k-n-under-hecke-less-0})}-\sum_{n=-\delta}^\infty \big(g_{k,n}(z)\big|\lambda\big) q_\tau^{np^l}\requiv{(\ref{G-k-expansion})} G_k(z,p^l\tau)\big|\lambda \pmod{p^{k-1}}
\end{align*}
as a double Fourier series in terms of $z$ and $\tau$. Also,
\begin{align*}
    G_k(z,p^l\tau)\eq \frac{g_k(z)\cdot g_{2-k}(p^l\tau)}{j(z)-j(p^l\tau)}\;\equiv\; \frac{g_k(z)\cdot g_{2-k}(\tau)^{p^l}}{j(z)-j(\tau)^{p^l}}\pmod{p}
\end{align*}
as a double Fourier series in terms of $z$ and $\tau$. We thus obtain the following.

\begin{proposition}
\label{U_p-action-on-G_k}
For a rational prime $p$, some $l\in\Z^+$, and a relation $\lambda$ for $S_k$, we have
\begin{align*}
    \frac{g_k(z)}{j(z)-j(\tau)}\bigg|\lambda U_{p^l}\;\equiv\; g_{2-k}(\tau)^{p^l-1}\cdot\frac{g_k(z)}{j(z)-j(\tau)^{p^l}}\bigg|\lambda\pmod{p}
\end{align*}
as a double Fourier series in terms of $z$ and $\tau$.
\end{proposition}

\subsection{Reducing to the case of weight $4$}

We will now prove that all the cases follow from the single case when $k=4$ and $\lambda=(1,0,0,\ldots)$.

Write
\begin{align*}
    \frac{g_k}{j-X}\bigg|\lambda\eq\sum_{n=1}^\infty P_{k,n,\lambda}(X)q^n
\end{align*}
and in particular,
\begin{align*}
    \frac{E_4}{j-X}\eq\sum_{n=1}^\infty P_{4,n}(X)q^n,
\end{align*}
for some polynomials $P_{k,n,\lambda},P_{4,n}\in\Z[X]$.

\begin{lemma}
\label{k=4-pth-coefficient}
For a rational prime $p$ and $l\in\Z^+$,
\begin{align*}
    P_{4,p^l}(j(\tau))\;\equiv\;(g_{-2}(\tau))^{p^l-1}\pmod{p}
\end{align*}
as $q_\tau$-series.
\end{lemma}
\begin{proof}
This follows from specializing \zcref{U_p-action-on-G_k} at $k=4$ and $\lambda=(1,0,0,\ldots)$.
\end{proof}

Now, by analyzing the potential poles and zeros, it is easy to check that there always exist $a,b\in\Z_{\geq0}$ such that
\begin{align*}
    \frac{g_{-2}(\tau)^{\frac{k-2}{2}}}{g_{2-k}(\tau)}\eq j(\tau)^a(j(\tau)-1728)^b.
\end{align*}
Then, \zcref{U_p-action-on-G_k} and \zcref{k=4-pth-coefficient} together imply that
\begin{align*}
    \bigg(\frac{g_{-2}(\tau)^{\frac{k-2}{2}}}{g_{2-k}(\tau)}\bigg)^{p^l-1}\cdot\frac{g_k(z)}{j(z)-j(\tau)}\bigg|\lambda U_{p^l}\;\equiv\; P_{4,p^l}(j(\tau))^{\frac{k-2}{2}}\cdot\frac{g_k(z)}{j(z)-j(\tau)^{p^l}}\bigg|\lambda\pmod{p}
\end{align*}
as a double Fourier series in terms of $z$ and $\tau$. In particular, taking the $n$-th coefficient of $q_z$ on both sides, one obtains that
\begin{align}
\label{G_k-and-G_4-equation}
    \big(j(\tau)^a(j(\tau)-1728)^b\big)^{p^l-1}P_{k,np^l,\lambda}(j(\tau))\;\equiv\; P_{4,p^l}(j(\tau))^{\frac{k-2}{2}}P_{k,n,\lambda}(j(\tau)^{p^l})\pmod{p}
\end{align}
as $q_\tau$-series, and hence also as polynomials in $j(\tau)$.

Now, let $L$ be a number field, $\fr{p}$ be a prime of $L$, and $c\in L$ with $v_\fr{p}(c)=0=v_\fr{p}(c-1728)$. Let $p$ be a rational prime lying under $\fr{p}$ and write $\N(\fr{p})=p^l$. Then, plugging in $j(\tau)=c$ in (\ref{G_k-and-G_4-equation}) yields that
\begin{align*}
    P_{k,np^l,\lambda}(c)\;\equiv\;P_{4,p^l}(c)^{\frac{k-2}{2}}P_{k,n,\lambda}(c)\pmod{\fr{p}}.
\end{align*}
In particular, we obtain the following.

\begin{lemma}
Let $k\in 2\Z_{\geq2}$, $L$ be a number field, and $C/L$ be an elliptic curve. Let $\lambda$ be a relation for $S_k$ and let $\fr{p}$ be a prime with $v_\fr{p}(j(C))=0=v_\fr{p}(j(C)-1728)$. Then, for all $n\in\Z^+$,
\begin{align*}
    a_{n\cdot\N(\fr{p})}(F_{k,C}|\lambda)\;\equiv\; a_{\N(\fr{p})}(F_{4,C})^{\frac{k-2}{2}}\thin a_n(F_{k,C}|\lambda)\pmod{\fr{p}}.
\end{align*}
\end{lemma}

From the lemma above, it suffices to show that
\begin{align*}
    a_{\N(\fr{p})}(F_{4,C})\;\equiv\;a_\fr{p}(C)^2\pmod{\fr{p}}
\end{align*}
in order to show \zcref{good-prime-theorem-number-field}.

\subsection{Hypergeometric reinterpretation of $a_{p^l}(F_{4,C})$}
\label{hypergeometric-subsection}

Now, we specialize to $a_{p^l}(F_{4,C})$. Specifically, we will prove \zcref{k=4-hypergeom-theorem}. For this part, it is more convenient to forget the dependence on the elliptic curves and simply write
\begin{align*}
    F_{4,c}\deq\frac{E_4}{j-c},
\end{align*}
for $c\in\ov{\Q}$ or $\ov{\Z_p}$.

Recall that $a_{p^l}(F_{4,c})=P_{4,p^l}(c)$ and
\begin{align*}
    P_{4,p^l}(j)\;\equiv\;g_{-2}^{p^l-1}\;\equiv\;(E_{10}/\Delta)^{p^l-1}\pmod{p}
\end{align*}
as $q$-series. In order to prove \zcref{k=4-hypergeom-theorem}, we will use the hypergeometric interpretation of $E_4$.

We first recall the definition of hypergeometric functions. Given a hypergeometric datum $(\bo{\alpha},\bo{\beta})$ with
\begin{align*}
    \bo{\alpha}&\eq(\alpha_1,\ldots,\alpha_{n-1},\alpha_n)\in\C^n \\
    \bo{\beta}&\eq(\beta_1,\ldots,\beta_{n-1})\in\C^{n-1},
\end{align*}
the associated hypergeometric function is defined as
\begin{align*}
    \pFq{n}{n-1}{\bo{\alpha}}{\bo{\beta}}{z}\deq\sum_{m=0}^\infty\frac{(\alpha_1)_m(\alpha_2)_m\cdots(\alpha_n)_m}{(\beta_1)_m(\beta_2)_m\cdots(\beta_{n-1})_m}\cdot \frac{z^m}{m!},
\end{align*}
and for $r\in\Z_{\geq0}$, the truncated hypergeometric sum is defined as
\begin{align*}
    \pFq{n}{n-1}{\bo{\alpha}}{\bo{\beta}}{z}_{r}\deq\sum_{m=0}^r\frac{(\alpha_1)_m(\alpha_2)_m\cdots(\alpha_n)_m}{(\beta_1)_m(\beta_2)_m\cdots(\beta_{n-1})_m}\cdot \frac{z^m}{m!},
\end{align*}
where $(a)_m:=\prod_{i=0}^{m-1}(a+i)$.

Recall the following identity by Fricke and Klein \cite[Chapter~5, Section~10]{fricke-elliptic-function}
\begin{align*}
    E_4^{1/4}\eq \pFq{2}{1}{\frac{1}{12},\frac{5}{12}}{1}{1728j^{-1}}
\end{align*}
and Clausen's formula
\begin{align*}
    \pFq{2}{1}{\frac{1}{12},\frac{5}{12}}{1}{t}^2\eq\pFq{3}{2}{\frac{1}{2},\frac{1}{6},\frac{5}{6}}{1,1}{t}.
\end{align*}
It follows that
\begin{align}
\label{g_(-2)-formula}
    \frac{E_{10}}{\Delta}\eq\bigg(\frac{E_4^3}{\Delta}\cdot\frac{E_6^2}{\Delta}\bigg)^{1/2}\cdot E_4^{-1/2}\eq\big(j(j-1728)\big)^{1/2}\cdot\pFq{3}{2}{\frac{1}{2},\frac{1}{6},\frac{5}{6}}{1,1}{1728j^{-1}}^{-1}.
\end{align}

Now, let $p\nmid 6$ be a rational prime and let $l\in\Z^+$. Write
\begin{align*}
    Q_n(X)\deq X^{n-1}P_{4,n}(X^{-1}).
\end{align*}
Then,
\begin{align*}
    Q_{p^l}(j^{-1})\eq j^{1-p^l} P_{4,p^l}(j)&\;\equiv\;j^{1-p^l}(E_{10}/\Delta)^{p^l-1} \\
    &\req{(\ref{g_(-2)-formula})} (1-1728j^{-1})^{\frac{p^l-1}{2}}\cdot\pFq{3}{2}{\frac{1}{2},\frac{1}{6},\frac{5}{6}}{1,1}{1728j^{-1}}^{1-p^l}\pmod{p}
\end{align*}
as $q$-series. Let $\alpha=j^{-1}=q+O(q^2)$. Then, 
\begin{align*}
    Q_{p^l}(\alpha)\;\equiv\;(1-1728\alpha)^{\frac{p^l-1}{2}}\cdot\pFq{3}{2}{\frac{1}{2},\frac{1}{6},\frac{5}{6}}{1,1}{1728\alpha}^{1-p^l}\pmod{p}
\end{align*}
as power series in $\alpha$. Write $B(\alpha)=\smallpFq{3}{2}{\frac{1}{2},\frac{1}{6},\frac{5}{6}}{1,1}{1728\alpha}$ for convenience. Then,
\begin{align*}
    Q_{p^l}(\alpha)\;\equiv\;(1-1728\alpha)^{\frac{p^l-1}{2}}B(\alpha)^{1-p^l}\;\equiv\;(1-1728\alpha)^{\frac{p^l-1}{2}}B(\alpha)\cdot (B^{-1})(\alpha^{p^l})\pmod{p},
\end{align*}
where $B^{-1}$ is the inverse of $B$ as a power series. Since $\deg Q_{p^l}=\deg P_{4,p^l}\leq p^l-1$ and $B(0)=1$, this implies that
\begin{align*}
    Q_{p^l}(\alpha)\;\equiv\;\Trun_{p^l-1,\alpha}\bigg((1-1728\alpha)^{\frac{p^l-1}{2}}\cdot\pFq{3}{2}{\frac{1}{2},\frac{1}{6},\frac{5}{6}}{1,1}{1728\alpha}\bigg)\pmod{p}.
\end{align*}
Here for a power series $R(t)=\sum_{n=0}^\infty a_nt^n$, we write
\begin{align*}
    \Trun_m(R)(t)\deq\sum_{n=0}^m a_nt^n\quad\text{ and }\quad\Trun_{m,t}\bigg(\sum_{n=0}^\infty a_nt^n\bigg)\deq\sum_{n=0}^m a_nt^n.
\end{align*}
For $p\nmid 6$, it is easy to check that $\frac{(\frac{1}{2})_r(\frac{1}{6})_r(\frac{5}{6})_r}{(r!)^3}$ is divisible by $p$ for $\frac{p^l-1}{6}<r\leq p^l-1$. From this, one can easily deduce that
\begin{align*}
    \Trun_{p^l-1,\alpha}\big((1-1728\alpha)^{\frac{p^l-1}{2}}\cdot B(\alpha)\big)
    \;\equiv\;(1-1728\alpha)^{\frac{p^l-1}{2}}\cdot\Trun_{p^l-1}(B)(\alpha)\pmod{p}.
\end{align*}
To summarize, we obtain that
\begin{align*}
    Q_{p^l}(\alpha)\;\equiv\;(1-1728\alpha)^{\frac{p^l-1}{2}}\cdot\pFq{3}{2}{\frac{1}{2},\frac{1}{6},\frac{5}{6}}{1,1}{1728\alpha}_{p^l-1}\pmod{p}
\end{align*}
as polynomials in $\alpha$.

Now, for $c\in\ov{\Z_p}$ with $v_p(c)=0$,
\begin{align*}
    a_{p^l}(F_{4,c})\eq P_{4,p^l}(c)&\eq c^{p^l-1} Q_{p^l}(c^{-1}) \\
    &\;\equiv\;c^{p^l-1}(1-1728c^{-1})^{\frac{p^l-1}{2}}\cdot\pFq{3}{2}{\frac{1}{2},\frac{1}{6},\frac{5}{6}}{1,1}{1728c^{-1}}_{p^l-1} \\
    &\;\equiv\;\big(c(c-1728)\big)^{\frac{p^l-1}{2}}\cdot\sum_{n=0}^{p^l-1}\frac{(6m)!}{(m!)^3(3m)!}\thin c^{-m}\pmod{p},
\end{align*}
which proves \zcref{k=4-hypergeom-theorem}.

\subsection{Proof of \zcref{good-prime-theorem-number-field}}

Let $L$ be a number field, $C/L$ be an elliptic curve, and $\fr{p}$ be a good prime of $C$ with $v_\fr{p}(j(C))=0=v_\fr{p}(j(C)-1728)$. Let $\F_\fr{p}$ denote the residue field of $\fr{p}$ and let $p$ be the rational prime lying under $\fr{p}$. For $a\in L$ with $v_\fr{p}(a)\geq 0$, define a quadratic character
\begin{align*}
    \bigg(\frac{a}{\fr{p}}\bigg)\eq\begin{cases}
        1&\text{ if $a$ is a square in $\F_\fr{p}^\times$,} \\
        -1&\text{ if $a$ is not a square in $\F_\fr{p}^\times$,} \\
        0&\text{ if $a=0$ in $\F_\fr{p}$.}
    \end{cases}
\end{align*}
It is easy to check that if $\fr{p}\nmid 2$, then
\begin{align*}
    \bigg(\frac{a}{\fr{p}}\bigg)\eq a^{\frac{\N(\fr{p})-1}{2}}\hspace{0.2cm}\text{ in }\F_\fr{p}.
\end{align*}
By \zcref{k=4-hypergeom-theorem},
\begin{align*}
    a_{\N(\fr{p})}(F_{4,C})\;\equiv\;\bigg(\frac{j(C)(j(C)-1728)}{\fr{p}}\bigg)\cdot\pFq{3}{2}{\frac{1}{2},\frac{1}{6},\frac{5}{6}}{1,1}{\frac{1728}{j(C)}}_{\N(\fr{p})-1}\pmod{p}.
\end{align*}
It thus suffices to show that
\begin{align*}
    a_\fr{p}(C)^2\;\equiv\;\bigg(\frac{j(C)(j(C)-1728)}{\fr{p}}\bigg)\cdot\pFq{3}{2}{\frac{1}{2},\frac{1}{6},\frac{5}{6}}{1,1}{\frac{1728}{j(C)}}_{\N(\fr{p})-1}\pmod{\fr{p}}.
\end{align*}
This congruence should follow from the hypergeometric machinery, though an explicit reference seems hard to find. Here we write a sketch of the proof and refer to the expository note \cite{3f2-congruence} for an explicit proof of this congruence.

Write $c=j(C)$ and consider the elliptic curve
\begin{align*}
    C_1: \; y^2+xy\eq x^3-\frac{t}{432}
\end{align*}
with $t=\frac{1-\sqrt{1-1728/c}}{2}$ (so that $j(C_1)=c$). It is known that for a prime $\fr{P}$ of $L(\sqrt{1-1728/c})$ that is a good prime of $C_1$,
\begin{align*}
    a_\fr{P}(C_1)\;\equiv\;\pFq{2}{1}{\frac{1}{6},\frac{5}{6}}{1}{t}_{\N(\fr{P})-1}\pmod{\fr{P}}.
\end{align*}
Now, one can relate $a_\fr{p}(C)$ and $a_\fr{P}(C_1)$ via the fact that $j(C)=j(C_1)$ and use an appropriate version of the truncated Clausen's formula to obtain the desired congruence.

\section{Proof of \zcref{cm-middle-supercongruence-intro}}
\label{shimura-magnetic-section}

We will now prove \zcref{cm-middle-supercongruence-intro}, i.e. the supercongruences and the $\frac{k-2}{2}$-magnetic property of~$G_{k,D}^{(k/2)}$.

Recall that for $k\in 2\Z_{\geq2}$, 
\begin{align*}
    G_k(z,\tau)\eq\frac{g_k(z)\cdot g_{2-k}(\tau)}{j(z)-j(\tau)}.
\end{align*}
For a discriminant $D<0$, define
\begin{align}
    \cl{G}_{k,D}(z)\deq\frac{2}{w_D}\sum_{Q\in\cl{Q}_D^\prim}(\partial^{\frac{k-2}{2}}_\tau G_k)(z,\alpha_Q).
\end{align}
As before, since $(\partial^{\frac{k-2}{2}}_\tau G_k)(z,\tau)$ is of weight $0$ with respect to $\tau$, $\cl{G}_{k,D}(z)$ is independent of the choice of~$\cl{Q}_D^\prim$.

\begin{theorem}
\label{magnetic-theorem}
Let $k\in 2\Z_{\geq2}$ and $\lambda$ be a relation for $S_k$. Let $D<0$ be a discriminant and write $D=A^2D_0$ for some $A\in\Z$ and some fundamental discriminant $D_0<0$. Let
\begin{align*}
    \wt{\cl{G}}_{k,D}\eq\begin{cases}
            |D_0|^{-\frac{k}{4}}\cl{G}_{k,D}&\text{ if }k\equiv0\pmod{4},\\
            |D_0|^{\frac{k-2}{4}}\cl{G}_{k,D}&\text{ if }k\equiv 2\pmod{4}.
    \end{cases}
\end{align*}
Then, $\wt{\cl{G}}_{k,D}$ satisfies the following. 
\begin{enumerate}[label=\rm{(\arabic*)}]
    \item $\wt{\cl{G}}_{k,D}\in\Z[[q]]$ and $\wt{\cl{G}}_{k,D}|\lambda$ is $\frac{k-2}{2}$-magnetic, i.e., for all $n\in\Z^+$,
    \begin{align*}
        n^{\frac{k-2}{2}}\thin\big|\thin a_n(\wt{\cl{G}}_{k,D}|\lambda).
    \end{align*}
    \item Let $p$ be a prime with $p\nmid A$. Then, for all $n,l\in\Z^+$,
    \begin{align*}
    a_{np^l}(\wt{\cl{G}}_{k,D}|\lambda)\;\equiv\;\big(\big(\tfrac{D}{p}\big)p\big)^{\frac{k-2}{2}}a_{np^{l-1}}(\wt{\cl{G}}_{k,D}|\lambda)\pmod{p^{(k-1)l}}.
    \end{align*}
\end{enumerate}
\end{theorem}

\begin{remark}
This theorem should be viewed as a natural generalization of Li--Neururer \cite[Theorem~1.5]{li-neururer} and Pa\c{s}ol--Zudilin \cite[Theorems~1~and~2]{pasol-zudilin}. Indeed, the proof uses the same method: one considers the half-integral weight preimage of $\cl{G}_{k,D}$ under the Shimura lift, analyzes the $U_p$ action on the preimage, and transfers the properties of the preimage to the supercongruences and the magnetic property of $\cl{G}_{k,D}$.
\end{remark}

\subsection{The Shimura lift}

Let $s\in\Z$, $f\in M_{s+1/2}^{!,+}$, and $d_0$ be a fundamental discriminant $d_0$ with $(-1)^sd_0>0$. Define the \emph{$d_0$-th Shimura lift} of $f$ by
\begin{align*}
    \cl{S}_{d_0}(f)\deq\frac{1}{2}\cdot L\big(1-s,(\tfrac{d_0}{\cdot})\big)\thin a_0(f)+\sum_{n=1}^\infty\bigg(\sum_{m\divides n}\big(\tfrac{d_0}{m}\big)\thin m^{s-1}\thin a_{|d_0|n^2/m^2}(f)\bigg)q^n.
\end{align*}

Now, let $d$ be a discriminant and $d_0$ be a fundamental discriminant with $dd_0<0$. For an $\SL_2(\Z)$-invariant function $f(\tau)$ on $\clh$, recall the following definition from \cite{duke-jenkins-basis}
\begin{align}
\label{tr-def-eqn}
    \Tr_{d,d_0}(f)\deq\sum_{Q\in\cl{Q}_{dd_0}}\frac{1}{w_Q}\cdot \chi(Q)f(\alpha_Q),
\end{align}
where for each $Q(X,Y)=aX^2+bXY+cY^2\in\cl{Q}_{dd_0}$, we define
\begin{align*}
    \chi(Q)\eq\chi_{d,d_0}(Q)\deq\begin{cases}
        \big(\frac{d_0}{r}\big)&\text{ if $(a,b,c,d_0)=1$ and $Q$ represents $r$ with $(r,d_0)=1$,} \\
        0&\text{ if $(a,b,c,d_0)>1$.}
    \end{cases}
\end{align*}

Now, let $s\in\Z_{\geq2}$, $d$ be a discriminant, and $d_0$ be a fundamental discriminant with $(-1)^sd_0>0$ and $dd_0<0$. Write $2s=k'+12\delta$ for some $k'\in\{0,4,6,8,10,14\}$ and $\delta\in\Z$. We define
\begin{align*}
    h_{s,d,d_0}\deq \sum_{n=0}^\delta a_n\big(\cl{S}_{d_0} f_{s+1/2,|d|}\big)\cdot g_{2s,-n}.
\end{align*}

\begin{lemma}
\label{h_d-lemma}
Let $s\in\Z_{\geq2}$, $d$ be a discriminant, and $d_0$ be a fundamental discriminant with $(-1)^sd_0>0$ and $dd_0<0$. Then, $h_{s,d,d_0}$ has integer coefficients and lies in $S_{2s}$.
\end{lemma}
\begin{proof}
The integrality of $h_{s,d,d_0}$ simply follows from the integrality of $f_{s+1/2,|d|}$ and $g_{2s,-n}$. Now, since $s\in\Z_{\geq2}$ and $|d|>0$, it is easy to check that $a_0(f_{s+1/2,|d|})=0$, so $a_0\big(\cl{S}_{d_0} f_{s+1/2,|d|}\big)=0$. Hence, $h_{s,d,d_0}\in S_{2s}$.
\end{proof}

\begin{proposition}
\label{shimura-preimage}
Let $s\in\Z_{\geq2}$, $d$ be a discriminant, and $d_0$ be a fundamental discriminant with $(-1)^sd_0>0$ and $dd_0<0$. Then,
\begin{align*}
    \big(\cl{S}_{d_0} f_{s+1/2,|d|}\big)(z)-h_{s,d,d_0}(z)\eq-(-1)^{\lfloor\frac{s-1}{2}\rfloor}\thin|d|^{-\frac{s}{2}}\thin|d_0|^{\frac{s-1}{2}}\;\Tr_{d,d_0}\big((\partial^{s-1}_\tau G_{2s})(z,\tau)\big),
\end{align*}
where $\Tr_{d,d_0}$ is taken with respect to $\tau$.
\end{proposition}
\begin{proof}
See \cite[Proposition~10]{duke-jenkins-basis}.
\end{proof}

\subsection{The Hecke actions}

We will now analyze the actions of the Hecke operators on $f_{s+1/2,m}$ and transfer them to modular forms of integer weights. 

\begin{lemma}
	\zlabel{half-int-integrality-lemma}
	Let $s\in\Z_{\geq2}$, $m\in\Z^+$ with $(-1)^{s-1}m\equiv 0,1\pmod{4}$, and $p$ be a prime such that
	\begin{enumerate}[label=\rm{(\arabic*)}]
		\item $p^2\nmid m$, or
		\item if $p=2$ and $4\divides m$, then $(-1)^{s-1}\thin\frac{m}{4}\equiv2,3\pmod{4}$.
	\end{enumerate}
	Let
	\begin{align*}
		g_0&\deq f_{s+1/2,m} \\
		g_1&\deq g_0|T_{p,s+1/2}-\Big(\tfrac{(-1)^{s-1}m}{p}\Big)p^{s-1}\cdot g_0 \\
		g_{i+1}&\deq g_i|T_{p,s+1/2}-p^{2s-1}\cdot g_{i-1}
	\end{align*}
	Then, for all $i\geq 0$,
	\begin{align*}
		g_i\eq p^{(2s-1)i}q^{-mp^{2i}}+O(q)\in\Z[[q,q^{-1}]].
	\end{align*}
	In particular, $g_i-p^{(2s-1)i}\cdot f_{s+1/2,mp^{2i}}$ has integer coefficients and lies in $S_{s+1/2}^+$.
\end{lemma}
\begin{proof}
	The case when $i=0$ follows from the definition of $f_{s+1/2,m}$.  For $i=1$, it is easy to check that the assumptions on $p$ and $m$ imply that
	\begin{align*}
		g_1&\eq g_0|T_{p,s+1/2}-\Big(\tfrac{(-1)^{s-1}m}{p}\Big)p^{s-1}\cdot g_0 \\
		&\eq\left(p^{2s-1}q^{-mp^2}+\Big(\tfrac{(-1)^{s-1}m}{p}\Big)p^{s-1}q^{-m}+O(q)\right)-\left(\Big(\tfrac{(-1)^{s-1}m}{p}\Big)p^{s-1}q^{-m}+O(q)\right) \\
		&\eq p^{2s-1}q^{-mp^2}+O(q),
	\end{align*}
	from which the result follows. Now, for $i\geq 1$, it is easy to check that by the induction hypothesis,
	\begin{align*}
		g_{i+1}\eq g_i|T_{p,s+1/2}-p^{2s-1}\cdot g_{i-1}\eq p^{(2s-1)(i+1)}q^{-mp^{2i+2}}+O(q).
	\end{align*}
	The result then follows.
\end{proof}

Now, recall that the Shimura lift is Hecke-equivariant. That is, for $s\in\Z$, $f\in M_{s+1/2}^{!,+}$, and a rational prime $p$,
\begin{align*}
	\cl{S}_{d_0}(f|T_{p,s+1/2})\eq\cl{S}_{d_0}(f)|T_{p,2s},
\end{align*}
where $d_0$ is some fundamental discriminant with $(-1)^sd_0>0$.

\begin{proposition}
	\zlabel{p^2-nmid-m}
	Let $s\in\Z_{\geq2}$ and let $d_0$ be a fundamental discriminant with \mbox{$(-1)^sd_0>0$}. Let $m\in\Z^+$ with $(-1)^{s-1}m\equiv0,1\pmod{4}$ and let $p$ be a prime such that
	\begin{enumerate}[label=\rm{(\arabic*)}]
		\item $p^2\nmid m$, or
		\item if $p=2$ and $4\divides m$, then $(-1)^{s-1}\thin\frac{m}{4}\equiv2,3\pmod{4}$.
	\end{enumerate}
	Let
	\begin{align*}
		F_m\deq\cl{S}_{d_0}f_{s+1/2,m}
	\end{align*}
	and let $\lambda$ be a relation for $S_{2s}$. Then, for all $n,l\in\Z^+$,
	\begin{align*}
		a_{np^l}(F_m|\lambda)\;\equiv\;\Big(\tfrac{(-1)^{s-1}m}{p}\Big)p^{s-1}\cdot a_{np^{l-1}}(F_m|\lambda)\pmod{p^{(2s-1)l}}.
	\end{align*}
	In particular, for all $n,l\in\Z^+$,
	\begin{align*}
		a_{np^l}(F_m|\lambda)\;\equiv\;0\pmod{p^{(s-1)l}}.
	\end{align*}
\end{proposition}
\begin{proof}
	Let $g_i\in M_{s+1/2}^{!,+}$ be as in \zcref{half-int-integrality-lemma} and let $G_i=\cl{S}_{d_0}g_i$. In particular, 
	\begin{align*}
		G_0&\eq F_m \\
		G_1&\eq G_0|T_{p,2s}-\Big(\tfrac{(-1)^{s-1}m}{p}\Big)p^{s-1}\cdot G_0 \\
		G_{i+1}&\eq G_i|T_{p,2s}-p^{2s-1}\cdot G_{i-1}.
	\end{align*}
	Let $\wt{G}_i=G_i|\lambda$. Then, since $\lambda$ commutes with $T_{p,2s}$, we have
	\begin{align*}
		\wt{G}_{i+1}&\eq \wt{G}_i|T_{p,2s}-p^{2s-1}\cdot \wt{G}_{i-1}.
	\end{align*}
	Write $H_i=\wt{G}_i|U_p-p^{2s-1}\wt{G}_{i-1}$. Then,
	\begin{align*}
		H_{i+1}\eq \wt{G}_{i+1}|U_p-p^{2s-1}\wt{G}_i\eq \left(\wt{G}_i|(U_p+p^{2s-1}V_p)-p^{2s-1}\wt{G}_{i-1}\right)\hspace{-0.08cm}\big|U_p-p^{2s-1} \wt{G}_i\eq H_i|U_p.
	\end{align*}
	It follows that
	\begin{align*}
		H_1|U_p^{l-1}\eq H_l.
	\end{align*}
	Now, since the Shimura lift preserves cusp forms and $\lambda$ annihilates cusp forms, it follows from \zcref{half-int-integrality-lemma} that
	\begin{align*}
		\wt{G}_i\eq p^{(2s-1)i}(\cl{S}_{d_0}f_{s+1/2,mp^{2i}})|\lambda\;\equiv\;0\pmod{p^{(2s-1)i}}.
	\end{align*}
	Hence,
	\begin{align}
		\label{H_l-congruence}
		H_l\eq \wt{G}_l|U_p-p^{2s-1}\wt{G}_{l-1}\;\equiv\; \wt{G}_l|T_{p,2s}-p^{2s-1}\wt{G}_{l-1}\;\equiv\;\wt{G}_{l+1}\;\equiv\;0\pmod{p^{(2s-1)(l+1)}}.
	\end{align}
	Now,
	\begin{align*}
		F_m|\lambda U_p-\Big(\tfrac{(-1)^{s-1}m}{p}\Big)p^{s-1}\cdot F_m|\lambda &\eq \wt{G}_0|U_p-\Big(\tfrac{(-1)^{s-1}m}{p}\Big)p^{s-1} \wt{G}_0  \\
		&\eq \wt{G}_1 - p^{2s-1}\wt{G}_0|V_p\;\equiv\;0\pmod{p^{2s-1}},
	\end{align*}
	and for $l\geq 2$,
	\begin{align*}
		F_m|\lambda U_p^l-\Big(\tfrac{(-1)^{s-1}m}{p}\Big)p^{s-1}\cdot F_m|\lambda U_p^{l-1}&\eq \wt{G}_1|U_p^{l-1} - p^{2s-1}\wt{G}_0|U_p^{l-2} \\
		&\eq  H_1|U_p^{l-2} \\
		&\eq H_{l-1} \;\equiv\;0\pmod{p^{(2s-1)l}}.
	\end{align*}
	The result then follows.
\end{proof}

\begin{proposition}
	\zlabel{p-partial-magnetic}
	Let $s\in\Z_{\geq2}$ and let $d_0$ be a fundamental discriminant with \mbox{$(-1)^sd_0>0$}. Let $m\in\Z^+$ with $(-1)^{s-1}m\equiv0,1\pmod{4}$ and let $p$ be a prime such that
	\begin{enumerate}[label=\rm{(\arabic*)}]
		\item $p^2\nmid m$, or
		\item if $p=2$ and $4\divides m$, then $(-1)^{s-1}\thin\frac{m}{4}\equiv2,3\pmod{4}$.
	\end{enumerate}
	Let $t\in\Z_{\geq 0}$ and let
	\begin{align*}
		F_{mp^{2t}}\deq\cl{S}_{d_0}f_{s+1/2,mp^{2t}}.
	\end{align*}
	Let $\lambda$ be a relation for $S_{2s}$. Then, for all $n,l\in\Z^+$,
	\begin{align*}
		a_{np^l}(F_{mp^{2t}}|\lambda)\;\equiv\;\Big(\tfrac{(-1)^{s-1}m}{p}\Big)p^{s-1}\cdot a_{np^{l-1}}(F_{mp^{2t}}|\lambda)\pmod{p^{(2s-1)(l-t)}}.
	\end{align*}
	In particular, for all $n,l\in\Z^+$,
	\begin{align*}
		a_{np^l}(F_{mp^{2t}}|\lambda)\;\equiv\;0\pmod{p^{(s-1)(l-t)}}.
	\end{align*}
\end{proposition}
\begin{proof}
	We will prove this by induction on $t$. Note that the base case when $t=0$ follows from \zcref{p^2-nmid-m}. Adopt the notation $G_i$ and $\wt{G}_i=G_i|\lambda$ with respect to $m$ as in the proof of \zcref{p^2-nmid-m}. Since the Shimura lift preserves cusp forms and $\lambda$ annihilates cusp forms, it follows from \zcref{half-int-integrality-lemma} that
	\begin{align*}
		\wt{G}_i\eq p^{(2s-1)i}(\cl{S}_{d_0}f_{s+1/2,mp^{2i}})|\lambda\eq p^{(2s-1)i}F_{mp^{2i}}|\lambda.
	\end{align*}
	Let $H_i=\wt{G}_i|U_p-p^{2s-1}\wt{G}_{i-1}$ as before. Then, 
	\begin{align*}
		\wt{G}_t|U_p^l\eq(H_t+p^{2s-1}\wt{G}_{t-1})|U_p^{l-1}\eq H_{t+l-1}+p^{2s-1}\wt{G}_{t-1}|U_p^{l-1}.
	\end{align*}
	By (\ref{H_l-congruence}),
	\begin{align*}
		H_{t+l-1}\;\equiv\;0\pmod{p^{(2s-1)(t+l)}}.
	\end{align*}
	Hence, 
	\begin{align*}
		F_{mp^{2t}}|\lambda U_p^l\eq p^{-(2s-1)t}H_{t+l-1}+F_{mp^{2t-2}}|\lambda U_p^{l-1}\;\equiv\; F_{mp^{2t-2}}|\lambda U_p^{l-1}\pmod{p^{(2s-1)l}},
	\end{align*}
	so
	\begin{align*}
		&F_{mp^{2t}}|\lambda U_p^l-\Big(\tfrac{(-1)^{s-1}m}{p}\Big)p^{s-1}F_{mp^{2t}}|\lambda U_p^{l-1}\;\equiv\; \\
		&\hspace{2cm}F_{mp^{2t-2}}|\lambda U_p^{l-1}-\Big(\tfrac{(-1)^{s-1}m}{p}\Big)p^{s-1}F_{mp^{2t-2}}|\lambda U_p^{l-2}\pmod{p^{(2s-1)l}}.
	\end{align*}
	The result now follows by induction.
\end{proof}

\begin{corollary}
	\zlabel{magnetic-prop}
	Let $s\in\Z_{\geq2}$ and let $d_0$ be a fundamental discriminant with \mbox{$(-1)^sd_0>0$}. Let $m\in\Z^+$ with $(-1)^{s-1}m\equiv0,1\pmod{4}$ and let $A$ be the largest integer such that $A^2\divides m$ and $(-1)^{s-1}m/A^2\equiv0,1\pmod{4}$. Let
	\begin{align*}
		F\deq A^{s-1}\cl{S}_{d_0}f_{s+1/2,m}
	\end{align*}
	and let $\lambda$ be a relation for $S_{2s}$. Then, $F|\lambda$ is $(s-1)$-magnetic, i.e., $n^{s-1}\divides a_n(F|\lambda)$ for all $n\in\Z^+$.
\end{corollary}
\begin{proof}
	Fix a prime $p$. It suffices to prove that
	\begin{align*}
		a_{np^l}(F|\lambda)\;\equiv\;0\pmod{p^{(s-1)l}}
	\end{align*}
	for all $n,l\in\Z^+$. Let $t=v_p(A)$ and let $m'=m/p^{2t}$. Then, $m'$ satisfies that $(-1)^{s-1}m'\equiv0,1\pmod{4}$. We will prove that the pair $(p,m')$ satisfies the assumptions in \zcref{p-partial-magnetic}.
	
	Suppose that $p^2|m'$. Then, $(Ap)^2|m$. If $p\geq 3$, then $(-1)^{s-1}m/(Ap)^2\equiv0,1\pmod{4}$, contradicting the maximality of $A$. Hence, $p=2$. Write $A'=A/p^{2t}$. By the maximality of~$A$,
	\begin{align*}
		(-1)^{s-1}\frac{m}{(2A)^2}\eq(-1)^{s-1}\frac{m'}{4}\cdot\frac{1}{A'^2}\;\equiv\;2,3\pmod{4}.
	\end{align*}
	Since $2\nmid A'$, this implies that $(-1)^{s-1}m'/4\equiv2,3\pmod{4}$. Hence, the pair $(p,m')$ satisfies the assumptions in \zcref{p-partial-magnetic}.
	
	Now, applying \zcref{p-partial-magnetic} with $m'p^{2t}$ yields that
	\begin{align*}
		a_{np^l}\big(F_{m'p^{2t}}|\lambda\big)\;\equiv\;0\pmod{p^{(s-1)(l-t)}},
	\end{align*}
	where $F_{m'p^{2t}}=\cl{S}_{d_0}f_{s+1/2,m}$. The result then follows since $F=A^{s-1}F_{m'p^{2t}}$ and $p^t|A$.
\end{proof}

\subsection{The case when $k/2$ is even}
Adopt all the notation in \zcref{magnetic-theorem}. Assume that $s=k/2$ is even. As $s$ is even,  one can apply \zcref{shimura-preimage} with $d_0=1$ and $d=D$ to obtain that
\begin{align}
	\big(\cl{S}_1f_{s+1/2,-D}\big)(z)-h_D(z)\eq (-1)^{\frac{s}{2}}\thin|D|^{-\frac{s}{2}}\thin\Tr_{D,1}\big((\partial_\tau^{s-1}G_{2s})(z,\tau)\big),
\end{align}
where $h_D:=h_{s,D,1}\in S_{2s}$ has integer coefficients by \zcref{h_d-lemma}. Since $d_0=1$, we have that $\chi(Q)=1$ for all $Q\in\cl{Q}_D$. Hence, for an $\SL_2(\Z)$-invariant function $f(\tau)$ on $\clh$,
\begin{align}
	\label{tr-D1-f-eqn}
	\Tr_{D,1}(f)\eq\sum_{Q\in\cl{Q}_D}\frac{1}{w_Q}f(\alpha_Q).
\end{align}
Write $D=A^2D_0$ for some $A\in\Z^+$ and some fundamental discriminant $D_0<0$. Then,
\begin{align}
	\label{clQ-D-decomp-eqn}
	\cl{Q}_D\eq\bigsqcup_{A'\divides A}\frac{A}{A'}\cdot \cl{Q}_{A'^2D_0}^\prim,
\end{align}
where $\frac{A}{A'}\cdot \cl{Q}_{A'^2D_0}^\prim=\{\frac{A}{A'}\cdot Q\mid Q\in\cl{Q}_{A'^2D_0}^\prim\}$. One can thus further decompose (\ref{tr-D1-f-eqn}) as
\begin{align*}
	\Tr_{D,1}(f)\eq\sum_{A'\divides A}\;\sum_{Q\in\cl{Q}_{A'^2D_0}^\prim}\frac{2}{w_{A'^2D_0}}f(\alpha_Q),
\end{align*}
where we use that $w_{A'^2D_0}=2w_Q$ for all $Q\in\cl{Q}_{A'^2D_0}^\prim$.

To summarize, we obtain that
\begin{align*}
	\big(\cl{S}_1f_{s+1/2,-D}\big)(z)-h_D(z)&\eq (-1)^{\frac{s}{2}}\thin|D|^{-\frac{s}{2}}\thin\sum_{A'\divides A}\;\sum_{Q\in\cl{Q}_{A'^2D_0}^\prim}\frac{2}{w_{A'^2D_0}}\thin(\partial_\tau^{s-1}G_{2s})(z,\alpha_Q) \\
	&\eq(-1)^{\frac{s}{2}}\thin|D|^{-\frac{s}{2}}\thin\sum_{A'\divides A}\cl{G}_{2s,A'^2D_0}(z).
\end{align*}
By the M\"{o}bius inversion formula,
\begin{align*}
	\cl{G}_{2s,A^2D_0}&\eq\sum_{A'\divides A}\mu\big(\tfrac{A}{A'}\big)(-1)^{\frac{s}{2}}\thin|A'^2D_0|^{\frac{s}{2}}(\cl{S}_1f_{s+1/2,-A'^2D_0}-h_{A'^2D_0}) \\
	&\eq(-1)^{\frac{s}{2}}|D_0|^{\frac{s}{2}}\sum_{A'\divides A}\mu\big(\tfrac{A}{A'}\big)A'^s(\cl{S}_1f_{s+1/2,-A'^2D_0}-h_{A'^2D_0}).
\end{align*}
Now, each $\cl{S}_1f_{s+1/2,-A'^2D_0}$ and $h_{A'^2D_0}$ have integer coefficients, so it follows immediately that $\wt{\cl{G}}_{k,D}=|D_0|^{-\frac{k}{4}}\cl{G}_{k,D}=|D_0|^{-\frac{s}{2}}\cl{G}_{2s,A^2D_0}$ has integer coefficients. Since each $h_{A'^2D_0}\in S_{2s}$ and $\lambda$ annihilates $S_{2s}$, we obtain that
\begin{align}
	\wt{\cl{G}}_{2s,A^2D_0}|\lambda\eq(-1)^{\frac{s}{2}}\sum_{A'\divides A}\mu\big(\tfrac{A}{A'}\big)A'^s\thin\cl{S}_1f_{s+1/2,-A'^2D_0}|\lambda.
\end{align}
As $s$ is even, each $A'^{s-1}\thin\cl{S}_1f_{s+1/2,-A'^2D_0}|\lambda$ is $(s-1)$-magnetic by \zcref{magnetic-prop}, so it follows that $\wt{\cl{G}}_{k,D}|\lambda$ is $\frac{k-2}{2}$-magnetic.

For the supercongruences, let $p$ be a prime with $p\nmid A$. Then, the pair $(p,m=-A'^2D_0)$ satisfies the assumption in \zcref{p^2-nmid-m}. Since $s$ is even, it thus follows that each $F_{-A'^2D_0}:=\cl{S}_1f_{s+1/2,-A'^2D_0}$ satisfies that for all $n,l\in\Z^+$,
\begin{align*}
	a_{np^l}\big(F_{-A'^2D_0}|\lambda\big)\;\equiv\;\big(\tfrac{D_0}{p}\big)p^{\frac{k}{2}-1}\thin a_{np^{l-1}}\big(F_{-A'^2D_0}|\lambda\big)\pmod{p^{(k-1)l}}.
\end{align*}
The supercongruence of $\wt{\cl{G}}_{k,D}|\lambda$ then follows.

\subsection{The case when $k/2$ is odd}
Adopt all the notation in \zcref{magnetic-theorem}. Assume that $s=k/2$ is odd. Write $D=A^2D_0$ for some $A\in\Z^+$ and some fundamental discriminant $D_0<0$ as before. As $s$ is odd, one can apply \zcref{shimura-preimage} with $d_0=D_0$ and $d=A^2$ to obtain that
\begin{align}
	\big(\cl{S}_{D_0}f_{s+1/2,A^2}\big)(z)-h_D(z)\eq (-1)^{\frac{s+1}{2}} A^{-s}\thin|D_0|^{\frac{s-1}{2}}\thin\Tr_{A^2,D_0}\big((\partial_\tau^{s-1}G_{2s})(z,\tau)\big),
\end{align}
where $h_D:=h_{s,A^2,D_0}\in S_{2s}$ has integer coefficients by \zcref{h_d-lemma}.
In this case, by (\ref{clQ-D-decomp-eqn}) and \zcref{genus-character-vanish-lemma}, we have that for an $\SL_2(\Z)$-invariant function $f(\tau)$ on $\clh$,
\begin{align*}
	\Tr_{A^2,D_0}(f)\eq\sum_{Q\in\cl{Q}_D}\frac{1}{w_Q}\chi(Q)f(\alpha_Q)\eq\sum_{A'\divides A}\;\sum_{Q\in\cl{Q}_{A'^2D_0}^\prim}\frac{2}{w_{A'^2D_0}}\Big(\tfrac{D_0}{A/A'}\Big)\thin f(\alpha_Q).
\end{align*}
To summarize, we obtain that
\begin{align*}
	\big(\cl{S}_{D_0}f_{s+1/2,A^2}\big)(z) - h_D(z)&\eq (-1)^{\frac{s+1}{2}} A^{-s}\thin|D_0|^{\frac{s-1}{2}}\sum_{A'\divides A}\;\sum_{Q\in\cl{Q}_{A'^2D_0}^\prim}\hspace{-0.2cm}\frac{2}{w_{A'^2D_0}}\Big(\tfrac{D_0}{A/A'}\Big)\thin (\partial_\tau^{s-1}G_{2s})(z,\alpha_Q) \\
	&\eq(-1)^{\frac{s+1}{2}} A^{-s}\thin|D_0|^{\frac{s-1}{2}}\sum_{A'\divides A}\Big(\tfrac{D_0}{A/A'}\Big)\cl{G}_{2s,A'^2D_0}(z).
\end{align*}
Hence,
\begin{align*}
	A^s(\cl{S}_{D_0}f_{s+1/2,A^2}-h_D)\eq(-1)^{\frac{s+1}{2}}|D_0|^{\frac{s-1}{2}}\sum_{A'\divides A}\Big(\tfrac{D_0}{A/A'}\Big)\cl{G}_{2s,A'^2D_0}.
\end{align*}
Now, both $\cl{S}_{D_0}f_{s+1/2,A^2}$ and $h_D$ have integer coefficients, so it follows by induction (on $A$) that $\wt{\cl{G}}_{k,D}=|D_0|^{\frac{k-2}{4}}\cl{G}_{k,D}=|D_0|^{\frac{s-1}{2}}\cl{G}_{2s,A^2D_0}$ has integer coefficients. Similarly as before, we obtain that
\begin{align}
	A^s\thin\cl{S}_{D_0}f_{s+1/2,A^2}|\lambda\eq(-1)^{\frac{s+1}{2}}\sum_{A'\divides A}\Big(\tfrac{D_0}{A/A'}\Big)\wt{\cl{G}}_{2s,A'^2D_0}|\lambda.
\end{align}
As $s$ is odd, $A^{s-1}\cl{S}_{D_0}f_{s+1/2,A^2}|\lambda$ is $(s-1)$-magnetic by \zcref{magnetic-prop}, so it follows by induction that $\wt{\cl{G}}_{k,D}|\lambda$ is $\frac{k-2}{2}$-magnetic.

For the supercongruences, let $p$ be a prime with $p\nmid A$. Then, the pair $(p,m=A^2)$ satisfies the assumption in \zcref{p^2-nmid-m}. Since $s$ is odd, it thus follows that each $F_{A^2}:=\cl{S}_{D_0}f_{s+1/2,A^2}$ satisfies that for all $n,l\in\Z^+$,
\begin{align*}
	a_{np^l}\big(F_{A^2}|\lambda\big)\;\equiv\;p^{\frac{k}{2}-1}\thin a_{np^{l-1}}\big(F_{A^2}|\lambda\big)\pmod{p^{(k-1)l}}.
\end{align*}
The supercongruence of $\wt{\cl{G}}_{k,D}|\lambda$ then follows by induction.

The following is a lemma which should follow from the theory of the genus characters of quadratic forms. We still include a proof for completeness. 

\begin{lemma}
	\zlabel{genus-character-vanish-lemma}
	Let $D=A^2D_0<0$ be a discriminant with $A\in\Z^+$ and $D_0<0$ being a fundamental discriminant. Let $Q\in\cl{Q}_D$ and suppose that $Q$ is $\SL_2(\Z)$-equivalent to $A_Q\cdot Q'$ for some $A_Q\in\Z^+$ with $A_Q\divides A$ and $Q'\in\cl{Q}_{D/A_Q^2}^\prim$. Then, $\chi_{A^2,D_0}(Q)=\big(\frac{D_0}{A_Q}\big)$.
\end{lemma}
\begin{proof}
	Since $Q'$ is primitive, there exists a prime $p'$  with $(p',D)=1$ such that $p'=Q'(x',y')$ for some $x',y'\in\Z$. Then,
	\begin{align*}
		\chi_{A^2,D_0}(Q)\eq\Big(\tfrac{D_0}{A_Q\cdot p'}\Big)\eq\Big(\tfrac{D_0}{A_Q}\Big)\Big(\tfrac{D_0}{p'}\Big),
	\end{align*}
	so it suffices to prove that $\big(\frac{D_0}{p'}\big)=1$. Since $\gcd(x',y')=1$ (as $p'=Q'(x',y')$ is a prime), there exist $u,v\in\Z$ such that $x'u-y'v=1$. Then,
	\begin{align*}
		Q'(x'X+vY,y'X+uY)\eq p'X^2+BXY+CY^2
	\end{align*}
	for some $B,C\in\Z$. In particular,
	\begin{align*}
		D_0\cdot(A/A_Q)^2\eq D/A_Q^2\eq B^2-4p'C,
	\end{align*}
	so $D_0$ is a square modulo $p'$, i.e., $\big(\frac{D_0}{p'}\big)=1$.
\end{proof}

\bibliographystyle{amsalpha}
\bibliography{ref}

\end{document}